\documentclass[11pt]{amsart}

\usepackage{amssymb,amsmath}
\usepackage{enumitem}
\usepackage{graphicx}
\usepackage{verbatim}
\usepackage{multirow}
\usepackage{framed}
\usepackage{color}

\newtheorem{theorem}{Theorem}[section]

\newtheorem{lemma}{Lemma}[section]

\newtheorem{proposition}[theorem]{Proposition}

\newtheorem{claim}{Claim}
\newtheorem{remark}{Remark}[section]

\newtheorem*{TheoremA}{Theorem A}
\newtheorem*{TheoremB}{Theorem B}

\def \l{\lambda}

\def \om{\omega}

\newcommand{\p}{\Phi}
\def \R{\mathbb{R}}

\def \M{\mathcal{M}}

\title[Thermodynamics via Inducing]{Thermodynamics via Inducing}
\author{Farruh Shahidi}
\address{Department of Mathematics \\ Pennsylvania State University \\ University Park, PA 16802, USA}
\email{fus144@psu.edu}
\author{Agnieszka Zelerowicz}
\address{Department of Mathematics \\ Pennsylvania State University \\ University Park, PA 16802, USA}
\email{axz157@psu.edu}

\begin{document}

\maketitle

\date{\today}

\begin{abstract}

We consider continuous maps $f:X\to X$ on compact metric spaces admitting inducing schemes of hyperbolic type introduced in \cite{ind} as well as the induced maps $\tilde{f}:\tilde{X}\to\tilde{X}$ and the associated tower maps $\hat{f}:\hat{X} \to \hat {X}$.
For a certain class of potential functions $\varphi$ on $X$, which includes all H\"older continuous functions, we establish thermodynamic formalism for each of the above three systems and we describe some relations between the corresponding equilibrium measures. Furthermore we study ergodic properties of these equilibrium measures including the Bernoulli property, decay of correlations, and the Central Limit Theorem (CLT). Finally, we prove analyticity of the pressure function for the three systems.
\end{abstract}

\section*{Introduction}\label{sec:intro}
\let\thefootnote\relax\footnote{The authors are partially supported by NSF grant
DMS-1400027.}

Celebrated works of Sinai, Ruelle, and Bowen establish thermodynamic formalism for uniformly hyperbolic systems.
Their main result claims that every H\"older continuous potential has a unique equilibrium measure. From the statistical physics point of view the statement about uniqueness implies absence of phase transitions. In addition, the equilibrium measure is known to have the Bernoulli property, to satisfy the Central Limit Theorem, and to have exponential decay of correlations. 
Those results apply to the family of geometric $t$-potentials $\varphi_t=-t \log | \det(Df_{|E^u})|$ 
(where $E^u$ denotes the unstable subspace) thus producing the one parameter family of equilibrium measures including such famous measures as Sinai-Ruelle-Bowen (SRB) measure and the measure of maximal entropy (MME).
We point out that the requirement on the potential to be H\"older continuous is crucial: indeed, one can
construct a family of potentials that are smooth everywhere except for a single point where they are not H\"older continuous, which admits a phase transition (see \cite{PZ}).

The natural next step is to study systems with weaker than uniform hyperbolicity. 
In this case even some H\"older continuous potential functions may exhibit phase transitions.
For example, this can be observed in the famous
Manneville-Pomeau map
(which is a one-dimensional expanding map with an indifferent fixed point) with respect to the
family of geometric $t$-potentials: at $t=1$ a phase transition occurred since the potential $\varphi_1$ has more than one equilibrium measure.
Moreover, such systems may not allow symbolic representation by a subshift of finite type.

Currently, one
of the most advanced methods to study non-uniformly hyperbolic systems is to use inducing maps on some appropriately chosen domains.
To this end, in this paper we consider maps with inducing schemes of hyperbolic type, which were introduced by Pesin, Senti, Zhang in \cite{ind}.
The principal feature of  
such maps is that the induced map possesses a countable generating partition allowing a symbolic representation of the induced map by the full shift on a countable set of states, thus bringing powerful tools of statistical physics into the study. In particular, using results of Aaronson-Denker \cite{AaD},  Mauldin-Urbanski \cite{MaulUrb}, Ruelle \cite{Rue}, Sarig \cite{Sar03}, and others one can construct Gibbs and equilibrium measures corresponding to H\"older continuous potentials with respect to standard metric in the space of two-sided sequences. Furthermore, one can obtain a sufficiently complete description of ergodic properties of these measures including ergodicity, the Bernoulli property, exponential rate of mixing and analyticity of the pressure function. From the statistical physics point of view, the latter means the absence of phase transitions, i.e. uniqueness of equilibrium states.

Examples of maps admitting inducing schemes of hyperbolic type include Young diffeomorphisms.
This is a broad class of diffeomorphisms which contains such systems as Billiard dynamical systems and H\'enon-type maps (see \cite{you}). 
In fact, recent work of Climenhaga, Luzzatto, and Pesin \cite{CLP} suggests that existence of a Young tower is a common phenomenon in smooth 
non-uniformly hyperbolic dynamics. 
 
In order to establish ergodic properties of a given system via inducing, one should consider an intermediate system - the tower map - which is an abstract, simplified model of the original map
and can be represented symbolically. 
Thus the way to effect thermodynamics of the original map is to first study  thermodynamics of the corresponding inducing system, then ``lift'' equilibrium measures to the tower and finally, ``project'' these measures to obtain the desired equilibrium measures for the original system. 
Difficulties arise
in studying ergodic properties of measures obtained by transferring them from one system to another.
 
Many remarkable results were obtained using Young tower techniques. Young \cite{you}, \cite{you2}, used tower construction
construct the Sinai-Ruelle-Bowen (SRB) 
measures for many important non-uniformly hyperbolic systems.
In addition, Pesin, Senti, and Zhang \cite{PSZ} proved the existence and uniqueness of equilibrium measures for maps with inducing schemes of hyperbolic type with respect to a certain, rather broad, class of potential functions.

 While there is an impressive body of work on Young towers, many results on the relations between the three systems are still missing and in our paper we present a systematic and rather complete study of the three systems: the original map $(X,f)$, the induced map $(\tilde{X},\tilde{f})$, and the tower map $(\hat{X},\hat{f})$, (see Section \ref{sec:tow}).
 We analyze relations among the corresponding equilibrium measures and their ergodic properties. We establish decay of correlations, the Central Limit Theorem, and the Bernoulli property for the three systems with respect to their corresponding equilibrium measures. Finally, we prove analyticity of the pressure function for the three systems. While the results on the Central Limit Theorem and bounds for the decay of correlations can be deduced from
 previous results, the one on the Bernoulli property and the analyticity of the pressure function are completely new.



The class of potential functions  $\varphi$ considered in this paper is defined
by a  set of conditions imposed on the induced potential $\tilde{\varphi}:\tilde{X}\to \tilde{X}$ (see Section \ref{sec:eqmeasures} for definitions and Section \ref{sec:Ps} for the list of conditions). We choose this class of potentials for two main reasons. First, this class is sufficiently large and includes all H\"older continuous functions on $X$. Second, our conditions on potential functions guarantee existence and uniqueness of equilibrium measures for the induced potential.


 

 The main results of this paper are as follows. Starting with the induced system 
$(\tilde{X},\tilde{f})$ we consider the class of potential functions on $X$ for which the corresponding induced potential satisfies \textbf{(P1)-(P4)} (see Section \ref{sec:Ps}). Then results in \cite{ind} guarantee existence and uniqueness of equilibrium measures for $(\tilde{X},\tilde{f})$ and the original map $(X, f)$. We then show how to obtain equilibrium measures for the tower map $(\hat{X},\hat{f})$. 
We also describe relations between unique equilibrium measures for the three systems (see Statement (1) of Theorem \ref{thm: tower} and Statements (1) and (2) of Theorem \ref{thm: manifold}).  

Our next goal is to obtain correlation estimates and the CLT for the map $(X,f)$ and for the tower map $(\hat{X},\hat{f})$. 
This is done using results in \cite{you2} and \cite{mel}.

One of the main contributions of this paper is the Bernoulli property of the map $(X,f)$ and of the tower map $(\hat{X},\hat{f})$. 
In fact, we show the following general result in a setting of symbolic dynamical systems.

Let $(S_A^{\mathbb{Z}},\sigma)$ be a topologically mixing  Markov shift on a countable set of states $S$. 
Choose a state $s\in S$ and a potential $\varphi : [s] \to \mathbb{R}$. Assume that $\varphi$ is locally H\"older continuous with respect to the induced shift $\bar{\sigma}$ on $[s]$, and that $P_G(\varphi)<\infty,$ where $P_{G}(\varphi)$ is the Gurevich pressure of $\varphi$ (see Section \ref{sec:prel}).
Let $\mu$ be the unique ergodic equilibrium measure for $\varphi$. 
Denote by $\hat{\mu}$ the measure on $S_A^{\mathbb{Z}}$ which is the lifted measure from $\mu$.
We have the following.

\begin{TheoremA}
If $(S_A^{\mathbb{Z}},\sigma, \hat{\mu})$ is mixing, then it is Bernoulli.
\end{TheoremA}

Our proof (see Section \ref{proof:thm:tower}) uses Ornstein's theory as well as some ideas of Sarig in \cite{SarBernoulli}. We stress that while the result in \cite{SarBernoulli} considers equilibrium measures for locally H\"older potential functions, our result is stronger, as it only requires that the {\bf induced} potential is locally H\"older.
 
Finally, we prove that the pressure for the map $f$ and the tower map $\hat{f}$ is  real analytic. Our argument follows ideas of Sarig in \cite{Sar00}, and extends the result of Sarig to the case of hyperbolic towers.

In terms of applications our main goal is to build thermodynamics
for Young diffeomorphisms and the family of geometric $t$-potentials for $t$ in some interval that contains $[0,1)$ (see Section \ref{sec:appl}). 
Particular examples include the Katok map \cite{THK}, the  H\'enon map at the first bifurcation \cite{ind}, and the slow-down map of the Smale-Williams solenoid \cite{Zel}. 
We emphasize that the method involved is general and is believed to be applicable to a broader class of examples. 
We also remark that
while decay of correlations and the CLT for Young diffeomorphisms as mentioned above were studied by many authors (see \cite{Gou},\cite{Sar03}\cite{mel},\cite{you},\cite{you2}), the Bernoulli property and analyticity of the pressure are new results.

While our main goal is to study equilibrium measures corresponding to  a broad class of potential functions, our principal result on the Bernoulli property can be obtained in a greater generality. Namely, in line with our setting, we show that for any $\tilde{f}$-invariant ergodic measure $\tilde{\mu}$ for the induced system $(\tilde{X},\tilde{f}),$ certain conditions on $\tilde{\mu}$ guarantee that the corresponding lifted measure $\hat{\mu}$ which is invariant for $(\hat{X},\hat{f})$ and the projected measure $\mu$ which is invariant for $(X,f)$ both have the Bernoulli property (see Theorem \ref{thm:generalcase}).

The paper is organized as follows. In Section \ref{sec:main-def}, we define our main objects -- maps with inducing schemes and their associated towers. In Section \ref{sec:main-results} we state our main results. In particular, Theorem \ref{thm: induced} gives decay of correlations, CLT, and the Bernoulli property for the induced system and Theorems \ref{thm: tower} and \ref{thm: manifold} establish these properties for the tower map and for the original map with an inducing scheme respectively. 
Theorem \ref{thm:pressure} establishes analyticity of the pressure for maps with inducing schemes as well as the associated tower maps.
Theorem \ref{thm:generalcase} deals with general measures on the inducing domain which are invariant under the induced system. We provide conditions on these measures that guarantee the Bernoulli property for the tower map and the original system.
Section \ref{sec:appl} is devoted to applications of our results to Young diffeomorphisms. In Section \ref{sec:proofs} we prove our main theorems after introducing some notions from the theory of countable Markov shifts. 

\subsubsection*{Acknowledgments} We would like to thank our advisor, Yakov Pesin, for posing the problem and for valuable suggestions.

\section{Maps with inducing schemes and associated tower maps}\label{sec:main-def}
In this section we introduce a map with inducing scheme, the induced map, and the associated tower map.

\subsection{A map with an inducing scheme, $(X,f)$}
Let $f:X\to X$ be a continuous map on a compact metric space. Throughout the paper we assume that $f$ has finite topological entropy $h_{top}(f)<\infty.$\\

Given a countable collection of disjoint Borel sets $S=\{J\}$ and a positive integer-valued function $\tau:S\to\mathbb{N},$ we say that $f$ admits an \textit{inducing scheme of hyberbolic type} $\{S,\tau\}$ with  \textit{inducing domain} $\tilde{X}:=\bigcup_{J\in S}J$ and \textit{inducing time} $\tau: X\to \mathbb{N}$ defined by

$$\tau(x)=\begin{cases}
\tau(J),\ \ x\in J\\
0,\ \ \ \ \  \ x\notin \tilde{X}
\end{cases}$$
provided the following conditions \textbf{(I1)-(I2)} hold:\\

\textbf{(I1)} For any $J\in S$ one has

$$f^{\tau(J)}(J)\subset \tilde{X}\ \ and\ \  \bigcup\limits_{J\in S}f^{\tau(J)}(J)=\tilde{X}.$$

Moreover, $f^{\tau(J)}_{|J}$ can be extended to a homeomorphism of a neighborhood of $J;$

\textbf{(I2)} For every bi-infinite sequence $ \boldsymbol{\omega}=(\om_n)_{n\in\mathbb{Z}}\in S^{\mathbb{Z}}$ there exist a unique sequence $\boldsymbol{x}=\boldsymbol{x}(\boldsymbol{\om})=(x_n=x_n(\boldsymbol{\om}))_{n\in\mathbb{Z}}$ such that\\
(a) $x_n\in\bar{J}_{\om_n}$ and $f^{\tau(J_{\om_n})}(x_n)=x_{n+1};$\\
(b) if $x_n(\boldsymbol{\om})=x_n(\boldsymbol{\om'})$ for all $n\le 0$, then $\boldsymbol{\om}=\boldsymbol{\om'}.$

\subsection{The induced map $(\tilde{X},\tilde{f})$}\label{def: induced}

Condition $\textbf{(I1)}$ allows one to define the \textit{ induced map} $\tilde{f}:\tilde{X}\to \tilde{X}$ by
$$\tilde{f}_{|J}:=f^{\tau(J)}_{|J},\ \ \ J\in S.$$

For each $J\in S,$ the map $\tilde{f}_{|J}$ can be extended to the closure $\bar{J}$ producing a map $\tilde{f}:\bigcup \bar{J} \to \overline{\tilde{X}}$.\\

Condition $\textbf{(I2)}$ allows one to define the coding map $\pi: S^{\mathbb{Z}}\to\bigcup\bar{J}$ by
$$\pi(\boldsymbol{\om}):=x_0(\boldsymbol{\om}).$$
Let $\Omega=\{\boldsymbol{\om}\in S^{\mathbb{Z}}: x_n(\boldsymbol{\om})\in J_{\om_n} \ \text{for all}\  n\in\mathbb{Z}\}.$

\begin{proposition}\cite{ind}\label{conj}
The map $\pi$ defined above has the following properties:

\begin{enumerate}
\item $\pi$ is well defined, continuous, and for all $\boldsymbol{\om}\in S^{\mathbb{Z}}$ one has
$$\pi\circ\sigma(\boldsymbol{\om})=\tilde{f}\circ\pi(\boldsymbol{\om});$$

\item $\pi$ is one-to-one on $\Omega$ and $\pi(\Omega)=\tilde{X};$

\item  if $\pi(\boldsymbol{\om})=\pi(\boldsymbol{\om'})$ for some $\boldsymbol{\om}, \boldsymbol{\om'}\in\Omega$, then $\om_n=\om'_n$ for all $n\ge 0.$
\end{enumerate}
\end{proposition}

\hspace{8mm}

For $x,y\in \tilde{X}$ let $s(x,y)$ denote the smallest integer $n\geq 0$ for which
$$(\pi^{-1}(x))_{n}\neq (\pi^{-1}(y))_{n} \text{ or } (\pi^{-1}(x))_{-n}\neq (\pi^{-1}(y))_{-n}.$$
In literature $s(x,y)$ is called the \textit{separation time}.

\subsection{The tower map $(\hat{X},\hat{f})$}\label{sec:tow}

Consider $$ \hat{X}:= \left\{ (x,k) ~ | ~ x\in \tilde{X},~k\in\{0,1,\ldots,\tau(x)-1  \}   \right\} .$$

By the tower map for $f:X\to X$  we refer to a dynamical system $\hat{f}: \hat{X} \to \hat{X}$ given by:
$$ \hat{f}(\hat{x})=\hat{f}(x,k):= \left\{ \begin{matrix} (x,k+1) & \text{for }k<\tau(x)-1\\
(\tilde{f}(x),0) & \text{for }k=\tau(x)-1   \end{matrix} \right.  .$$
The map $\pi_0:\hat{X}\to X$ defined by $\pi_0(x,k)=f^k(x)$ satisfies the equation $f\circ\pi_0=\pi_0\circ\hat{f}.$
In general, $\pi_0$ need not to be either one-to-one or onto. However, the image $\pi_0(\hat{X})$ has full measure with respect to any ergodic measure
$\mu$ such that $\mu(\tilde{X})>0$. 

We end this section with defining a metric on $\hat{X}$. 
The separation time defined in the previous section can be extended to $\hat{X}$ by setting
$$
s((x,k),(y,l)) = \left\{ \begin{array}{lcr}
0 & \text{ if } & k\neq l \\
s(x,y) & \text{ if } & k= l, \\
\end{array} \right.
$$
where $s(x,y)$ was defined in the previous section.

Given $0<\theta<1$ we define a symbolic metric $d_{\theta}$ on $\hat{X}$ by setting
\begin{equation}\label{dtheta} 
d_{\theta}((x,k),(y,l))=\theta^{s((x,k),(y,l))}. 
 \end{equation}
 Throughout the paper by a class of H\"older continuous functions on $\hat{X}$ we mean the class of functions which are 
  H\"older continuous with respect to $d_{\theta}$.

We would like to stress that throughout the paper the symbol $\tilde{ ~ }$ is used to denote objects related to the induced system
and the symbol $\hat{ ~ }$ is used to denote objects related to the tower map. For example having a measure $\mu$ for the original system,
$\tilde{\mu}$ denotes the corresponding measure for the induced system, and $\hat{\mu}$ denotes the corresponding measure for the tower map.

\subsection{Lifting and projecting measures}\label{sec:relations}

We are going to use the following notation. Having a Borel transformation $T:D \to D$ on a metric space $D$ we denote by $\mathcal{M}(D,T)$
the space of $T$-invariant Borel probability measures on $D$. We denote the space of $T$-invariant ergodic Borel probability measures on $D$
by $Erg(D,T)$.

 For any $\tilde{\mu}\in \mathcal{M}(\tilde{X},\tilde{f})$ let
\begin{equation}\label{def:Qu}
 Q_{\tilde{\mu}}:= \int_{\tilde{X}} \tau d\tilde{\mu}.  
\end{equation}
If $Q_{\tilde{\mu}}<\infty$ we define the corresponding \textit{lifted measure} $\mu=\mathcal{L}(\tilde{\mu})$ on $X$ as follows:
for any $E\subset X$,
\begin{equation}\label{def:Lu} 
\mathcal{L}(\tilde{\mu})(E):= \frac{1}{Q_{\tilde{\mu}}}  \sum_{J\in S} \sum_{k=0}^{\tau(J)-1}\tilde{\mu}(f^{-k}(E)\cap J)  .  
\end{equation}

We can identify every point $(x,0)\in \hat{X}$ with a point $x\in \tilde{X}\subset X$. Observe that using this identification, and replacing $f$ with $\hat{f}$ in (\ref{def:Lu}), one obtains lifted measure $\hat{\mu}=\hat{\mathcal{L}}(\tilde{\mu})$ on $\hat{X}$.

It is clear that $\mathcal{L}(\tilde{\mu})\in \mathcal{M}(X,f)$ and $\hat{\mathcal{L}}(\tilde{\mu})\in \mathcal{M}(\hat{X},\hat{f})$.
In addition, if $\tilde{\mu}\in Erg(\tilde{X},\tilde{f})$, then $\mathcal{L}(\tilde{\mu})\in Erg(X,f)$ and $\hat{\mathcal{L}}(\tilde{\mu})\in Erg(\hat{X},\hat{f})$.

We consider the class of \textit{lifted} measures on X, 
$$\mathcal{M}_L(X,f):= \{ \mu\in \mathcal{M}(X,f) | \text{there is }\nu\in \mathcal{M}(\tilde{X},\tilde{f}) \text{ with }\mathcal{L}(\nu)=\mu  \}.$$
Similarly we define $\mathcal{M}_L(\hat{X},\hat{f})$.

Consider a measure $\hat{\mu}\in \mathcal{M}(\hat{X},\hat{f})$ with $\hat{\mu}(\tilde{X}\times \{0\})>0$. After identifying $\tilde{X}\times \{ 0 \}$ with $\tilde{X}$ it is 
a simple observation that every such measure induces a measure $\tilde{\mu}\in \mathcal{M}(\tilde{X},\tilde{f})$ defined by
$\tilde{\mu} :=\frac{1}{\hat{\mu}(\tilde{X})} \hat{\mu}_{|\tilde{X}}$.
In addition, 
$$Q_{\tilde{\mu}}= \int_{\tilde{X}} \tau d\tilde{\mu}= \frac{1}{\hat{\mu}(\tilde{X})}   \int_{\tilde{X}} \tau d \hat{\mu}=  \frac{1}{\hat{\mu}(\tilde{X})}  \hat{\mu} (\hat{X})        =\frac{1}{\hat{\mu}(\tilde{X})}.$$

This means that  $\mathcal{M}_L(\hat{X},\hat{f}) = \{   \hat{\mu}\in \mathcal{M}(\hat{X},\hat{f})|  \hat{\mu}(\tilde{X}\times\{0\}) >0    \}$.

Another consequence of this fact is that every measure $\mu\in\M_L(X,f)$ can be obtained as a \textit{projection} $\mathcal{R}(\hat{\mu})$ of a certain measure $\hat{\mu}\in \mathcal{M}(\hat{X},\hat{f})$. Namely,
$\mu =  \mathcal{L}(\tilde{\mu}) =    \mathcal{L}(\frac{1}{\hat{\mu}(\tilde{X})} \hat{\mu}_{|\tilde{X}}) =: \mathcal{R}(\hat{\mu}). $

The operations of lifting and projecting measures establish relations between invariant measures for the three dynamical systems.

\subsection{Equilibrium measures}\label{sec:eqmeasures}
In studying equilibrium measures using methods of inducing we are forced to consider only those measures which are lifted from
measures on the inducing domain. Therefore, we need to adjust the notion of an equilibrium measure that fits our setting.

Given a potential $\varphi: X \to \mathbb{R}$,
a measure $\mu_{\varphi}\in \mathcal{M}_L(X,f)$ is called an \textit{equilibrium measure} for $\varphi$
(in the space $\mathcal{M}_L(X,f)$ of \textit{lifted} measures) if
$$ P_L(\varphi):=\sup_{\mu\in \mathcal{M}_L(X,f)} \{ h_{\mu}(f) + \int_{X}  \varphi d\mu  \} 
=  h_{\mu_{\varphi}}(f) + \int_{X}  \varphi d\mu_{\varphi}. $$

We define the corresponding potential $\hat{\varphi}$ on the tower $\hat{X}$ as
$$\hat{\varphi}(x,k):= \varphi ( f^k(x)).$$
A measure $\mu_{\hat{\varphi}}\in \mathcal{M}_L(\hat{X},\hat{f})$ is called an \textit{equilibrium measure} for $\hat{\varphi}$
(in the space $\mathcal{M}_L(\hat{X},\hat{f})$ of \textit{lifted} measures) if
$$ P_L(\hat{\varphi}):=\sup_{\hat{\mu}\in \mathcal{M}_L(\hat{X},\hat{f})} \{ h_{\hat{\mu}}(\hat{f}) + \int_{\hat{X}}  \hat{\varphi} d\hat{\mu}  \} 
=  h_{\mu_{\hat{\varphi}}}(\hat{f}) + \int_{\hat{X}}  \hat{\varphi} d\mu_{\hat{\varphi}}. $$

We define the \textit{induced potential} $\tilde{\varphi}: \tilde{X} \to \mathbb{R}$ by
$$ \tilde{\varphi}(x):= \sum_{k=0}^{\tau(x)-1}\varphi(f^k(x)).  $$
A measure $\nu_{\tilde{\varphi}}\in \mathcal{M}(\tilde{X},\tilde{f})$ is called an \textit{equilibrium measure} for $\tilde{\varphi}$ if
$$  h_{\nu_{\tilde{\varphi}}}(\tilde{f}) + \int_{\tilde{X}}  \tilde{\varphi} d\nu_{\tilde{\varphi}}= \sup_{\tilde{\nu}\in \mathcal{M}(\tilde{X},\tilde{f})} \{ h_{\tilde{\nu}}(\tilde{f}) + \int_{\tilde{X}}  \tilde{\varphi} d\tilde{\nu}  \} . $$

\section{Main results}\label{sec:main-results}

We start by stating additional conditions on the inducing scheme and on potential functions.

\subsection{Conditions on the inducing scheme}\label{sec:Is}

Denote by $\sigma: S^{\mathbb{Z}}\rightarrow S^{\mathbb{Z}}$ the left full shift and let
$$\Omega:=\{\boldsymbol{\om}\in S^{\mathbb{Z}}: x_n(\boldsymbol{\om})\in J_{\om_n}\ for\ all\ n\in\mathbb{Z}\},$$
where $x_n$ and $\omega_n$ are as in Condition \textbf{(I2)}.

\textbf{(I3)}  The set $S^{\mathbb{Z}}\setminus\Omega$ supports no $\sigma$- invariant measure which gives positive weight to any open subset.

\textbf{(I4)} The map $\tilde{f}:\tilde{X}\to \tilde{X}$ given by
$\tilde{f}_{|J}:=f^{\tau(J)}_{|J}$ has at least one periodic point in $\tilde{X}.$

$\textbf{(I5)}\ \   gcd \{ \tau(J)  | J\in S \} =1,$\\
where gcd stands for the greatest common divisor. Condition \textbf{(I5)} is called the arithmetic condition.

\subsection{Conditions on potential functions} \label{sec:Ps} Let $\varphi:X\to \mathbb{R}$ be a potential function.
We express conditions on the potential $\varphi$ in terms of the induced potential:

$\textbf{(P1)}$ The induced potential $\tilde{\varphi}$ can be extended by continuity to a function on $\bar{J}$ for all $J\in S;$

$\textbf{(P2)}$ The induced potential $\tilde{\varphi}$ is \textit{locally H\"{older} continuous}, i.e. for any $n\geq 1$
$$ Var_n(\tilde{\varphi}):=Var_n(\tilde{\varphi}\circ\pi) \leq C\theta^n  $$
for some constants $C>0$ and $0<\theta<1$, where $\pi$ is the coding map defined in Section \ref{def: induced} and $Var_n$ is defined in Section \ref{sec:prel};

$\textbf{(P3)}$ $$\sum_{J\in S} \sup_{x\in J} \exp \tilde{\varphi}(x)<\infty;$$

The \textit{normalized induced potential}  $\varphi^+: \tilde{X} \to \mathbb{R}$  is given by
$$ \varphi^+: = \tilde{\varphi} -  P_L(\varphi)\tau.  $$

$\textbf{(P4)}$ there exists $\epsilon>0$ such that  

$$  \sum_{J\in S}\tau(J)\sup_{x\in J} \exp (\varphi^+(x) + \epsilon \tau(x)) <\infty.   $$

\subsection{Preliminaries}

Let $X$ be a measurable space and $T:X\to X$ a measurable invertible transformation preserving a measure $\mu$. For reader's convenience we recall some properties of the system $(X,T,\mu)$ which are of interest to us in the paper.\\

\begin{enumerate}
\item \textbf{The Bernoulli property.} We say that $(X,T,\mu)$ has the \emph{Bernoulli property} if it is metrically isomorphic to the Bernoulli shift 
$(Y^{\mathbb{Z}},\sigma,\nu^{\mathbb{Z}})$ associated to some Lebesgue space $(Y,\nu)$, where $\nu$ is metrically isomorphic to the Lebesgue measure on an interval together with at most countably many atoms.\\
\item \textbf{Decay of correlations} Let $\mathcal{H}_1$ and $\mathcal{H}_2$ be two classes of observables on $X$. For $h_1\in\mathcal{H}_1$ and $h_2\in\mathcal{H}_2$ define the correlation function 
$$
\text{Cor}_n(h_1\circ T^n, h_2):=\int h_1(T^n(x))h_2(x)\,d\mu -\int h_1(x)\,d\mu
\int h_2(x)\,d\mu.
$$
We say that $T$ has \emph{exponential decay of correlations} with respect to classes $\mathcal{H}_1$ and $\mathcal{H}_2$ if there exist $C>0, \ 0<\theta<1$ such that for any $h_1\in\mathcal{H}_1$, $h_2\in\mathcal{H}_2$ and any $n>0$
$$|\text{Cor}_n(h_1\circ T^n, h_2)|<C\theta^n$$
and  say that $T$ has \emph{polynomial decay of correlations}  if there exists $\gamma>0$ such that
$$
|\text{Cor}_n( h_1\circ T^n, h_2)|\le Cn^{-\gamma},
$$
where $C=C(h_1,h_2)>0$ is a constant.


\item \textbf{The Central Limit Theorem} We say that $T$ satisfies the \emph{Central Limit Theorem (CLT)} for a class $\mathcal{H}$ of observables on $X$ if there exists $\sigma>0$ such that for any $h\in\mathcal{H}$ with 
$\int h=0$ the sum
$$
\frac 1{\sqrt{n}}\sum\limits_{i=0}^{n-1}h(f^i(x))
$$ 
converges in law to a normal distribution $\textit{N}(0,\sigma)$.
\end{enumerate}

\subsection{Statements of main results}

In this section we establish existence, uniqueness, and describe ergodic properties of equilibrium measures for the three systems: $(X,f,\varphi)$,
$(\tilde{X},\tilde{f},\varphi^+)$, and $(\hat{X},\hat{f},\hat{\varphi})$. We denote the class of all H\"older continuous functions on $X$ with exponent $\alpha$ by  $C^\alpha(X).$
The sequence $\nu_{\varphi^+}(\tau>n):= \nu_{\varphi^+}(\{x\in \tilde{X}| \tau(x)>n\})$ is known as the \textit{tail} of the measure $\nu_{\varphi^+}.$ We say that the tail is \textit{exponential (polynomial)} if $\nu_{\varphi^+}(\tau>n)$ decays exponentially(polynomially). We first consider the induced system $(\tilde{X},\tilde{f}, \tilde{\varphi}).$

\begin{theorem}\label{thm: induced}
Let $\tilde{f}: \tilde{X} \to \tilde{X}$ be an induced map satisfying Condition \textbf{(I3)}. Assume that the induced potential $\tilde{\varphi}$ satisfies Conditions \textbf{(P1)-(P4)}. Then:

\begin{enumerate}
\item 	 There exists a unique equilibrium Gibbs measure $\nu_{\varphi^{+}}$ for $\varphi^{+};$

\item The map $\tilde{f}$ has exponential decay of correlations for the observables $h_1\in L^{\infty}(\tilde{X},\nu_{\varphi^{+}})$ and $h_2\in C^{\alpha}(\tilde{X});$

\item The map $\tilde{f}$ satisfies the CLT with respect to $\nu_{\varphi^{+}}$ and any observable $h\in C^{\alpha}(\tilde{X});$

\item The map $\tilde{f}$ has the Bernoulli property with respect to $\nu_{\varphi^{+}};$

\end{enumerate}
\end{theorem}
Statement (1) of Theorem \ref{thm: induced} is proved in \cite{ind}. 
We note that similar results in the symbolic setting were obtained by Mauldin-Urbanski \cite{MaulUrb} and Sarig \cite{Sar03}.
Statements (2) and (3) follow from the well known result by Ruelle \cite{Rue}.
Statement (4) follows from the result in \cite{Daon} (see also \cite{SarBernoulli}).

We now state a similar result for the associated tower map $(\hat{X},\hat{f}, \hat{\varphi})$.

\begin{theorem}\label{thm: tower} Let $\hat{f}:\hat{X} \to \hat{X}$ be the tower map corresponding to an inducing scheme satisfying Conditions \textbf{(I3)} and \textbf{(I4)}. Let 
$\hat{\varphi}: \hat{X} \to \mathbb{R}$ correspond to some potential $\varphi: X \to \mathbb{R}$ satisfying Conditions $\textbf{(P1)-(P4)}$. Then: 
\begin{enumerate}
\item there exists a unique $\hat{f}$-invariant ergodic equilibrium measure $\mu_{\hat{\varphi}}$ for $\hat{\varphi}$ and it is the lifted measure for $\nu_{\varphi^+}.$\\
Assume, in addition, that the tower satisfies Condition \textbf{(I5)}.
\item  If $(\hat{X},\hat{f},\mu_{\hat{\varphi}})$ is mixing,
then it has the Bernoulli property.
\item For $\hat{h}_1, \hat{h}_2\in C^{\alpha}(\hat{X})$ one has:\\
 (a) If $\nu_{\varphi^{+}}(\tau>n)=\mathcal{O}(\theta^n)$ for $0<\theta<1,$ then $\text{Cor}_n(\hat{h}_1\circ\hat{f}^n,\hat{h}_2)$ decays exponentially;\\
 (b) If $\nu_{\varphi^{+}}(\tau>n)=\mathcal{O}(\frac{1}{n^{\beta}}), \beta>1,$ then
$\text{Cor}_n(\hat{h}_1\circ\hat{f}^n,\hat{h}_2)=\mathcal{O}(\frac{1}{n^{\beta-1}})$; \\
  (c) If $\nu_{\varphi^{+}}(\tau>n)=\mathcal{O}(\frac{1}{n^{\beta}}), \beta>1$ and 
$\hat{h}_1,\ \hat{h}_2 $ supported on $\bigcup_{j=0}^{k}\tilde{X}\times\{j\}$ ($\tilde{X}=\tilde{X}\times\{0\}$) for some $k,$ then
\begin{equation}\label{Gouzel}
\text{Cor}_n(\hat{h}_1\circ\hat{f}^n,\hat{h}_2)=\sum_{N>n}^{\infty}\nu_{\varphi^{+}}( \tau(x)>N)\int_{\hat{X}}\hat{h}_1\,d\mu_{\hat{\varphi}}\int_{\hat{X}}\hat{h}_2\,d\mu_{\hat{\varphi}}+r_{\beta}(n),
\end{equation}
where $r_{\beta}(n)=\mathcal{O}(R_{\beta}(n))$ and 
$$
R_{\beta}(n)=
\begin{cases}
\frac{1}{n^{\beta}}&  \text{if } \beta>2,\\ 
\frac{\log n}{n^2} &  \text{if } \beta=2, \\ 
\frac{1}{n^{2\beta-2}}& \text{if } 1<\beta<2.
\end{cases}  
$$ 
Moreover, if $\int_{\hat{Y}}\hat{h}_2=0$, then 
$\text{Cor}_n(\hat{h}_1\circ\hat{f}^n,\hat{h}_2)=\mathcal{O}(\frac1{n^{\beta}})$;\\

(d)  If $\nu_{\varphi^{+}}(\tau>n)=\mathcal{O}(\frac{1}{n^{\beta}}), \beta>1$ and 
$\hat{h}_1,\hat{h}_2 $ supported on $\bigcup_{j=0}^{k}\tilde{X}\times\{j\}$ ($\tilde{X}=\tilde{X}\times\{0\}$) for some $k,$ then $\hat{f}$ satisfies the CLT with respect to $\mu_{\hat{\varphi}}.$
\end{enumerate}
\end{theorem}
We remark that Statements (3a)-(3d) of Theorem \ref{thm: tower} were shown in \cite{you2} for towers corresponding to inducing schemes of expanding type (that is, towers which can be modeled by one-sided Markov shifts on a finite or countable set of states) and measures satisfying certain conditions. Those results were later extended in \cite{mel} to the case of inducing schemes of hyperbolic type. We summarize those results in Proposition \ref{noequilibrium}. Then we show that local H\"older continuity of the induced potential $\tilde{\varphi}$ and the fact that $\mu_{\hat{\varphi}}$ is obtained by lifting the measure $\nu_{\varphi^{+}}$ ensures that we can apply Proposition \ref{noequilibrium}.\\

We now consider the original system $(X,f, \varphi)$. For $x,y\in \tilde{X}$ define the \textit{forward separation time} $s^+(x,y)$ to be the smallest positive integer $n\geq 0$ for which $\tilde{f}^n(x),\tilde{f}^n(y)$ lie in different partition elements of $\tilde{X}$ (compare with the quantity $s(x,y)$ defined in Section \ref{def: induced}). We define the \textit{backward separation time} $s^-(x,y)$ analogously, so that $s(x,y) = \min\{ s^+(x,y), s^-(x,y)\}$.

\begin{theorem}\label{thm: manifold}
Let $f: X \to X $ be a continuous map that admits an inducing scheme satisfying Conditions 
$\textbf{(I1)-(I4)}.$ Let  $\varphi : X \to \mathbb{R}$ be a potential that satisfies Conditions 
$\textbf{(P1)-(P4)}$. Then:
\begin{enumerate} 
\item there exists a unique $f$-invariant ergodic equilibrium measure $\mu_{\varphi}$ for 
$\varphi$ and it is the lifted measure for $\nu_{\varphi^+}.$ 
\item  $\mu_{\varphi}$ is the projection of $\mu_{\hat{\varphi}}$.\\

\noindent If the inducing scheme satisfies Condition \textbf{(I5)}, then:\\ 

\item If $(X,f,\mu_{\varphi})$ is mixing, then it has the Bernoulli property.
\item Assume, in addition, that $(X,f,\mu_{\varphi})$  satisfies the following conditions:\\
\textbf{(F1)} there exist $C\geq 1$ and $ 0<\gamma <1$ such that if $s^+(x,y)=\infty$, then 
$d(\tilde{f}^nx,\tilde{f}^ny)\leq C \gamma^n $, and if $s^-(x,y)=\infty$, then 
$d(\tilde{f}^nx,\tilde{f}^ny)\ge C \gamma^{(s(x,y)-n)}$.\\
\textbf{(F2)} there exist $C\geq 1$ such that for all $x,y\in \tilde{X}$ and $k< \min\{\tau(x),\tau(y)\},$ \\
$d(f^k(x), f^k(y))\le C\max\{d(\tilde{f}x, \tilde{f}y),d(x,y)\}$.  \\
Then for $h_1,  h_2\in C^{\alpha}(X)$ one has: 
\begin{itemize}
\item[(a)]  If $\nu_{\varphi^{+}}(\tau>n)=\mathcal{O}(\theta^n)$ for $0<\theta<1$ then $\text{Cor}_n(h_1\circ f^n, h_2)$ decays exponentially;

\item[(b)]  If $\nu_{\varphi^{+}}(\tau>n)=\mathcal{O}(\frac{1}{n^{\beta}}),\  \beta>1,$ then
$\text{Cor}_n(h_1\circ f^n, h_2)=\mathcal{O}(\frac{1}{n^{\beta-1}})$; 

\item[(c)] If $\nu_{\varphi^{+}}(\tau>n)=\mathcal{O}(\frac{1}{n^{\beta}}), \beta>1,$ then there exists a sequence of nested sets $Y_0\subset Y_1\subset\cdots$ in $X$ such that if $h_1, h_2$ are supported inside $Y_k$ for some $k\ge 0,$ then
\begin{equation}\label{Gouzel1}
\text{Cor}_n(h_1\circ f^n,h_2)=
\sum_{N>n}^{\infty}\nu_{\varphi^{+}}(\tau(x)>N)\int_{X}h_1\,d\mu_{\varphi}\int_{X}h_2\,d\mu_{\varphi}
+r_{\beta}(n),
\end{equation}
where $r_{\beta}(n)=\mathcal{O}(R_{\beta}(n))$ and 
$$
R_{\beta}(n)=
\begin{cases}
\frac{1}{n^{\beta}}&  \text{if } \beta>2,\\ 
\frac{\log n}{n^2} &  \text{if } \beta=2, \\ 
\frac{1}{n^{2\beta-2}}& \text{if } 1<\beta<2.
\end{cases}  
$$ 
Moreover, if $\int_{Y}h_2=0$, then 
$\text{Cor}_n(h_1\circ f^n,h_2)=\mathcal{O}(\frac1{n^{\beta}})$;
\item[(d)]  If $\nu_{\varphi^{+}}(\tau>n)=\mathcal{O}(\frac{1}{n^{\beta}}), \beta>1$ and 
 $h_1,\ h_2$ are supported inside $Y_k$ for some $k\ge 0,$ then $f$ satisfies the CLT with respect to $\mu_{\varphi}.$
 \end{itemize}
\end{enumerate}
\end{theorem}

Statement (1) of Theorem \ref{thm: manifold} is shown in \cite{ind}. Note that Statement (4a) was mentioned in Theorem 4.5 in \cite{ind} with reference to \cite{you}. However, Conditions \textbf{(I5)} and \textbf{(F1)-(F2)} were missing there, so Statement (4a) complements the part of Theorem 4.5 in \cite{ind} on exponential decay of correlations.
The remaining statements  are shown in Section \ref{proof:thm:manifold}.\\

\subsection{Analyticity of the pressure}

Let $f: X \to X $ be a continuous map that admits an inducing scheme satisfying Conditions $\textbf{(I1)-(I5)}$, $\tilde{f}: \tilde{X} \to \tilde{X}$ the induced map, and $\hat{f}:\hat{X} \to \hat{X}$ the tower map. 
Let $\varphi_1, ~\varphi_2 : X \to \mathbb{R}$ be two potential functions satisfying Condition $\textbf{(P1)}$.
We consider the corresponding potentials $\hat{\varphi}_1(x,k):= \varphi_1 ( f^k(x)),$ $\hat{\varphi}_2(x,k):= \varphi_2 ( f^k(x))$ on the tower $\hat{X}$.
We also consider
the induced potential 
$ \tilde{\varphi}_2(x):= \sum_{k=0}^{\tau(x)-1}\varphi_2(f^k(x)),  $
and the normalized induced potential
$ \varphi_1^+:= \tilde{\varphi}_1 - P_L(\varphi_1)\tau$ on $\tilde{X}$.

We have the following.

\begin{theorem}\label{thm:pressure}
Assume there exists $\epsilon_0>0$ such that the following holds for all $|t|<\epsilon_0$:

\begin{enumerate}
\item there exists $c\in \mathbb{R}$ such that $\varphi_1 + c$ and $\varphi_1 + t \varphi_2 + c$ satisfy $\textbf{(P3)}$,\\

\item  $\varphi_1 $ and $\varphi_1 + t |\varphi_2| $ satisfy $\textbf{(P4)}$, \\

\item  $\varphi_1 $ and $\varphi_2 $ satisfy $\textbf{(P2)}$. 
\end{enumerate}

Then for some $0<\epsilon<\epsilon_0$ the functions $t \to P_G(\varphi_1^+ + t \tilde{\varphi}_2 )$, $t \to P_L(\hat{\varphi}_1 + t \hat{\varphi}_2 )$, and $t \to P_L(\varphi_1 + t \varphi_2 )$ are real analytic on $(-\epsilon,\epsilon)$.
\end{theorem}

Here $P_G$ denotes the Gurevich pressure which we define in Section \ref{sec:prel}, while $P_L$ was defined in Section \ref{sec:eqmeasures}.
For the induced system the fact that  $t \to P_G(\varphi_1^+ + t \tilde{\varphi}_2 )$ is real analytic
follows from a result of Sarig \cite{Sar00}. Analyticity of the functions  $t \to P_L(\hat{\varphi}_1 + t \hat{\varphi}_2 )$ and $t \to P_L(\varphi_1 + t \varphi_2 )$ are shown in Section \ref{proof:thm:pressure}.

\subsection{Generalization}
Our results on the Bernoulli property are corollaries of a more general result, 
where measures in question are not necessarily equilibrium measures.
Let $\tilde{\mu}$ be an invariant ergodic measure for the induced map $\tilde{f}$ on $\tilde{X}$. Let $\mu$ and $\hat{\mu}$ be the corresponding lifted measures on $X$ and $\hat{X}$ respectively. Note that $\mu$ and $\hat{\mu}$ are invariant and ergodic with respect to $f$ and $\hat{f}.$ 

Given a string $(\omega_1 \ldots \omega_n) \in \mathbb{N}^{n}$ we denote
$$   J_{\omega_1 \ldots \omega_n}:= \{ x\in \tilde{X} | \tilde{f}^{k-1}(x) \in J_{\omega_k} \text{ for }k=1,\ldots,n    \}   .$$   
Given two strings $A=(a_1\ldots a_{n_1})$, $B=(b_1\ldots b_{n_2})$, $M\in \mathbb{N}$, and $k\geq M$ we define
$$ S(A,B,M,k) := \{ x\in    J_{A} \cap \tilde{f}^{-(n_1+M)}J_{B}  | \tau(\tilde{f}^{n_1}x) + \ldots + \tau(\tilde{f}^{n_1+M-1}x)=k   \}  .  $$

We impose a collection of conditions on $(\tilde{X},\tilde{f},\tilde{\mu}).$

\textbf{(B1)} For every finite sub-collection $S^*\subset S$ there exists a constant $C^*=C^*(S^*)>1$ such that for every
pair of strings $A=(a_1\ldots a_{n_1})$, $B=(b_1\ldots b_{n_2})$,  if $J_{a_{n_1}} \in S^*$ or $J_{b_1}\in S^*$, and $J_{AB}\neq \emptyset$, then
\begin{equation}\label{eqn:directproduct}
\frac{1}{C^*}\leq \frac{\tilde{\mu}(J_{AB})}{\tilde{\mu}(J_{A})\tilde{\mu}(J_{B})}\leq C^*.
\end{equation}

\textbf{(B2)} For any $\delta>0$ there exists $m_0\in\mathbb{N}$ such that for any $m\geq m_0$ and for any finite collection $\gamma$ of $m-$strings
$C=(c_1 \ldots c_m) \in \mathbb{N}^m$ there exists $K\in\mathbb{N}$ such that the following statement holds:

Let $C, C'\in \gamma$, $k>K$, and let $A=(a_1 \ldots a_{n_1})$, $B=(b_1 \ldots b_{n_2})$ be any two strings. Assume that $J_{AC},J_{C'B}\neq \emptyset$. Then
\begin{equation}\frac{e^{-\delta}}{Q_{\tilde{\mu}}}\leq \frac{\sum_{M=1}^{k} \tilde{\mu}   ( S(AC,C'B,M,k) )      }{   \tilde{\mu}(J_{AC})\tilde{\mu}(J_{C'B})   }\leq \frac{e^{\delta}}{Q_{\tilde{\mu}}},
\end{equation}
 where $Q_{\tilde{\mu}}$ is defined by (\ref{def:Qu}).

\begin{remark}
 Requirement \textbf{(B1)} means that the measure $\tilde{\mu}$ has direct product structure. Both Conditions, \textbf{(B1)} and \textbf{(B2)}, are stated in terms of the induced system $(\tilde{X},\tilde{f}),$ but it is worth mentioning that the numerator of the middle term in the inequality in Condition $\textbf{(B2)}$ is equal to $\tilde{\mu}(J_{AC}\cap \hat{f}^{-k}(J_{C'B}))$. 
\end{remark}

\begin{theorem}\label{thm:generalcase}
Let $(\tilde{X},\tilde{f})$ be an induced map and let $\tilde{\mu}$ be an invariant ergodic measure for $\tilde{f}$.
If $(\tilde{X},\tilde{f},\tilde{\mu})$ satisfies Conditions \textbf{(B1)} and \textbf{(B2)},
then the systems $(X,f,\mu)$ and $(\hat{X},\hat{f},\hat{\mu})$ have the Bernoulli property. 
\end{theorem}

In the view of Theorem \ref{thm:generalcase}, the proof of  Bernoulli property for the measures $\mu$ and $\hat{\mu}$ in Theorems
\ref{thm: tower} and  \ref{thm: manifold} reduces to verifying Conditions \textbf{(B1)}, and \textbf{(B2)} for $\tilde{\mu}$.

\section{Application: Thermodynamics of Young diffeomorphisms}\label{sec:appl} 
\subsection{Definition of Young diffeomorphisms}\label{sec:young} Consider a 
$C^{1+\epsilon}$ diffeomorphism $f:M\to M$ of a compact smooth Riemannian manifold $M$. Following \cite{you} we describe a collection of conditions on the map $f$. 

An embedded $C^1$-disk $\gamma\subset M$ is called an \emph{unstable disk} (respectively, a \emph{stable disk}) if for all $x,y\in\gamma$ we have that
$d(f^{-n}(x), f^{-n}(y))\to 0$ (respectively, $d(f^{n}(x), f^{n}(y))\to 0$) as $n\to+\infty$. A collection of embedded $C^1$ disks $\Gamma^u=\{\gamma^u\}$\label{sb:gamma} is called a \emph{continuous family of unstable disks} if there exists a homeomorphism
$\Phi: K^s\times D^u\to\cup\gamma^u$ satisfying:
\begin{itemize}
\item $K^s\subset M$ is a Borel subset and $D^u\subset \R^d$ is the closed unit disk for some $d<\dim M$;
\item $x\to\Phi|{\{x\}\times D^u}$ is a continuous map from $K^s$ to the space of $C^1$ embeddings of $D^u$ into $M$ which can be extended to a continuous map of the closure 
$\overline{K^s}$;
\item $\gamma^u=\Phi(\{x\}\times D^u)$ is an unstable disk.
\end{itemize}
A \emph{continuous family of stable disks} is defined similarly.

We allow the sets $K^s$ to be non-compact in order to deal with overlaps which appear in most known examples including the Katok map.

A set $\Lambda\subset M$ has \emph{hyperbolic product structure} if there exists a continuous family $\Gamma^u=\{\gamma^u\}$ of unstable disks $\gamma^u$ and a continuous family $\Gamma^s=\{\gamma^s\}$ of stable disks $\gamma^s$ such that
\begin{itemize}
\item $\text{dim }\gamma^s+\text{dim }\gamma^u=\text{dim } M$;
\item the $\gamma^u$-disks are transversal to $\gamma^s$-disks with an angle uniformly bounded away from $0$;
\item each $\gamma^u$-disks intersects each $\gamma^s$-disk at exactly one point;
\item $\Lambda=(\cup\gamma^u)\cap(\cup\gamma^s)$.
\end{itemize}

A subset $\Lambda_0\subset\Lambda$ is called an \emph{$s$-subset} if it has hyperbolic product structure and is defined by the same family $\Gamma^u$ of unstable disks as $\Lambda$ and a continuous subfamily 
$\Gamma_0^s\subset\Gamma^s$ of stable disks. A \emph{$u$-subset} is defined analogously.

We define the \textit{s-closure} $scl(\Lambda_0)$ of an \textit{s-subset} $\Lambda_0\subset \Lambda$ by
$$  scl(\Lambda_0):= \bigcup_{x\in \overline{\Lambda_0\cap\gamma^u}} \gamma^s(x)\cap \Lambda. $$
We define the \textit{u-closure} $ucl(\Lambda_1)$ of a given \textit{u-subset} $\Lambda_1\subset \Lambda$ similarly:
$$ ucl(\Lambda_1):= \bigcup_{x\in \overline{\Lambda_1\cap \gamma^s}} \gamma^u(x)\cap \Lambda.  $$ 
Assume the map $f$ satisfies the following conditions:

\begin{enumerate}
\item[(Y1)] There exists $\Lambda\subset M$ with hyperbolic product structure, a countable collection of continuous subfamilies $\Gamma_i^s\subset\Gamma^s$ of stable disks and positive integers $\tau_i$, $i\in\mathbb{N}$ such that the $s$-subsets 
\begin{equation}\label{Lambda}
\Lambda_i^s:=\bigcup_{\gamma\in\Gamma^s_i}\,\bigl(\gamma\cap \Lambda\bigr)\subset\Lambda
\end{equation}
are pairwise disjoint and satisfy:
\begin{enumerate}
\item \emph{invariance}: for every $x\in\Lambda_i^s$ 
$$
f^{\tau_i}(\gamma^s(x))\subset\gamma^{s}(f^{\tau_i}(x)), \,\, f^{\tau_i}(\gamma^u(x))\supset\gamma^u(f^{\tau_i}(x)),
$$
where $\gamma^{u,s}(x)$ denotes the (un)stable disk containing $x$;
\item \emph{Markov property}: $\Lambda_i^u:=f^{\tau_i}(\Lambda_i^s)$ is a $u$-subset of 
$\Lambda$ such that for all $x\in\Lambda_i^s$ 
$$
\begin{aligned}
f^{-\tau_i}(\gamma^s(f^{\tau_i}(x))\cap\Lambda_i^u)
&=\gamma^s(x)\cap \Lambda,\\
f^{\tau_i}(\gamma^u(x)\cap\Lambda_i^s)
&=\gamma^u(f^{\tau_i}(x))\cap \Lambda.
\end{aligned}
$$
\end{enumerate}
\end{enumerate}
\vspace*{.5cm}
\begin{enumerate}
\item[(Y2)] The sets $\Lambda_i^u$ are pairwise disjoint.
\end{enumerate}

For any $x\in \Lambda^s_i$ define the \emph{inducing time} by 
$\tau(x):=\tau_i$ and the \emph{induced map} $\tilde{f}: \bigcup_{i\in\mathbb{N}}\Lambda_i^s\to\Lambda$
by
$$
\tilde{f}|_{\Lambda_i^s}:=f^{\tau_i}|_{\Lambda_i^s}.$$
\begin{enumerate}
\item[(Y3)] There exists $0<a<1$ such that for any $i\in\mathbb{N}$ we have:
\begin{enumerate}
\item[(a)] For $x\in\Lambda_i^s$ and $y\in\gamma^s(x)$,
$$
d(\tilde{f}(x), \tilde{f}(y))\le a\, d(x,y);
$$
\item[(b)] For $x\in\Lambda^s_i$ and $y\in\gamma^u(x)\cap \Lambda_i^s$,
$$
d(x,y)\le a \, d(\tilde{f}(x), \tilde{f}(y)).
$$
\end{enumerate}
\end{enumerate}

\begin{enumerate}
\item[(Y4)]  $gcd \{ \tau_i  | i\in\mathbb{N} \} =1.$
\end{enumerate}

We remark that our definition of Young diffeomorphism differs from the original one (see \cite{you}, \cite{you2}). In particular we do not assume bounded distortion.
This condition will be required later to establish thermodynamics for a special
class of potential functions.

\subsection{Thermodynamics of Young diffeomorphisms}
 Let
$$ \tilde{X}:=\bigcap_{n=-\infty}^{\infty} ucl\left( \tilde{f}^n(\bigcup_{i\in\mathbb{N}}\Lambda_i^s) \right)  $$

be non-empty maximal $\tilde{f}$-invariant set contained in $ucl(\Lambda)$. Set

$$ S=\{\Lambda_i^s\cap \tilde{X}\}_{i\in\mathbb{N}} ~~~~~~\text{ and }~~~~~~~ \tau(\Lambda_i^s\cap \tilde{X})=\tau_i.   $$

It is shown in \cite{PSZ} that if a diffeomorphism $f$ satisfies (Y1)-(Y3), then the corresponding inducing scheme $\{S,\tau\}$ satisfies Conditions \textbf{(I1)}, \textbf{(I2)}, and \textbf{(I4)}.
Condition $(Y4)$ guarantees that the corresponding inducing scheme satisfies \textbf{(I5)}.
Condition \textbf{(I3)} has to be verified independently for a given Young diffeomorphism.
Observe in addition that Condition (Y3) implies Condition \textbf{(F1)} in Theorem \ref{thm: manifold}.

By Theorems \ref{thm: manifold} and \ref{thm:pressure}, we therefore have the following.

\begin{theorem}\label{appl: Young ergprop gen} Let $f:M\to M$ be a $C^{1+\epsilon}$ diffeomorphims of a compact smooth Riemannian manifold $M$ satisfying Conditions (Y1)-(Y3).
	Assume that the corresponding inducing scheme $\{S,\tau \}$ satisfies Condition \textbf{(I3)} and let $\varphi: M \to \mathbb{R}$ be a potential function
	satisfying Conditions \textbf{(P1)-(P4)}. Then:

\begin{enumerate}
	\item there exists a unique equilibrium measure $\mu_{\varphi}$ among all lifted measures for the potential $\varphi;$\\
	
\noindent Assume in addition that $f$ satisfies Condition (Y4). Then:	\\

	 \item if $(M,f,\mu_{\varphi})$ is mixing, then it has the Bernoulli property.
	 
\item Assume that there exists $K>0$ such that for any $i\in\mathbb{N}$ and
for $x,y\in \Lambda_i^s$ and $0\leq j \leq \tau_i$,
$$
d(f^j(x),f^j(y))\leq K \max\left\{ d(x,y), d(\tilde{f}(x),\tilde{f}(y))    \right\}  .
$$
Then for $h_1,  h_2\in C^{\alpha}(M)$ one has: 
	\begin{itemize}
	\item[(a)] If $\nu_{\varphi^{+}}(\tau>n)=\mathcal{O}(\theta^n)$ for $0<\theta<1$ then $\text{Cor}_n(h_1\circ f^n, h_2)$ decays exponentially;
	
	\item[(b)] If $\nu_{\varphi^{+}}(\tau>n)=\mathcal{O}(\frac{1}{n^{\beta}}),\  \beta>1,$ then
	$\text{Cor}_n(h_1\circ f^n, h_2)=\mathcal{O}(\frac{1}{n^{\beta-1}}).$ Consequently, $\mu_{\varphi}$ is mixing; 
	
	\item[(c)] If $\nu_{\varphi^{+}}(\tau>n)=\mathcal{O}(\frac{1}{n^{\beta}}), \beta>1,$ then there exists a sequence of nested sets $Y_0\subset Y_1\subset\cdots$ in $M$ such that if $h_1, h_2$ are supported inside $Y_k$ for some $k\ge 0,$ then
	\begin{equation}\label{Gouzel1young}
	\text{Cor}_n(h_1\circ f^n,h_2)=
	\sum_{N>n}^{\infty}\nu_{\varphi^{+}}(\tau(x)>N)\int_{X}h_1\,d\mu_{\varphi}\int_{X}h_2\,d\mu_{\varphi}
	+r_{\beta}(n),
	\end{equation}
	Moreover, if $\int_{Y}h_2=0$, then 
	$\text{Cor}_n(h_1\circ f^n,h_2)=\mathcal{O}(\frac1{n^{\beta}})$;
	\item[(d)] If $\nu_{\varphi^{+}}(\tau>n)=\mathcal{O}(\frac{1}{n^{\beta}}), \beta>1$ and 
	 $h_1,\ h_2$ are supported inside $Y_k$ for some $k\ge 0,$ then $f$ satisfies the CLT with respect to $\mu_{\varphi}.$
	 \end{itemize}
\item If $\varphi_1: M \to \mathbb{R}$ is a potential satisfying Conditions \textbf{(P1)}, \textbf{(P2)}, and such that
for some $\epsilon_0>0$, $c\in\mathbb{R}$, and all $|t|<\epsilon_0$ one has:
\begin{itemize}
\item[(a)]  $\varphi + t \varphi_1 + c$ satisfies Condition \textbf{(P3)},

\item[(b)] $\varphi + t |\varphi_1| $ satisfies Condition \textbf{(P4)},
\end{itemize}
	 then for some $0< \epsilon< \epsilon_0$ the function $t\to P_L(\varphi + t \varphi_1)$ is real analytic on $(-\epsilon,\epsilon)$.
	\end{enumerate}


\end{theorem}

\subsection{Geometric $t$-potentials}

 For $t\in\mathbb{R}$ consider the family of geometric $t$-potentials

$$
\varphi_t(x):=-t\log|Jac f(x)|.
$$
 
In order to apply our results to the family of geometric $t$-potentials, one has to verify Conditions \textbf{(P1)-(P4)}.
For this we need the following additional conditions.

\begin{enumerate}
\item[(Y5)] For every $\gamma^u\in\Gamma^u$ one has
$$
\mu_{\gamma^u}(\gamma^u\cap \Lambda)>0, \quad \mu_{\gamma^u}\left((\overline{\Lambda\setminus\cup\Lambda_i^s)\cap\gamma^u}\right)=0,
$$
where $\mu_{\gamma^u}$ is the leaf volume on $\gamma^u$.
\end{enumerate}

\begin{enumerate}
\item[(Y6)] There exists $\gamma^u\in\Gamma^u$ such that 
$$
\sum_{i=1}^\infty \tau_i \mu_{\gamma^u}(\Lambda_i^s\cap \gamma^u) <\infty.
$$
\end{enumerate}

For $x\in \Lambda$ let $Jac f(x)=\det |Df|_{E^u(x)}|$ and  $Jac \tilde{f}(x)=\det |D\tilde{f}|_{E^u(x)}|$ denote the Jacobian of $Df|_{E^u(x)}$ and $D\tilde{f}|_{E^u(x)}$ respectively. 
\begin{enumerate}
\item[(Y7)] There exist $c>0$ and $0<\kappa<1$ such that:
\begin{enumerate}
\item[(a)] For all $n\ge 0$, $x\in \tilde{f}^{-n}(\cup_{i\in\mathbb{N}}\Lambda^s_i)$ 
and $y\in\gamma^s(x)$ we have
\[
\left|\log\frac{Jac \tilde{f}(\tilde{f}^{n}(x))}{Jac \tilde{f}(\tilde{f}^{n}(y))}\right|\le c\kappa^n;
\]
\item[(b)] For any $i_0,\dots, i_n\in\mathbb{N}$, 
$\tilde{f}^k(x),\tilde{f}^k(y)\in\Lambda^s_{i_k}$ for $0\le k\le n$ and  
$y\in\gamma^u(x)$ we have
\[
\left|\log
\frac{ Jac \tilde{f}(\tilde{f}^{n-k}(x))}{Jac \tilde{f}(\tilde{f}^{n-k}(y))}\right|\le c\kappa^k.
\]
\end{enumerate}
\end{enumerate}

We also need the following estimate. 
\begin{enumerate}
\item[(Y8)] $	S_n:=\sharp\,\{\Lambda_i^s\colon \tau_i=n\}\le Ce^{hn}$

\noindent	where $C>0$ and $0<h<-\int \varphi_1d\mu_1.$ \\

\end{enumerate}


It is shown in \cite{PSZ} that if a Young diffeomorphism satisfies Conditions (Y1)-(Y3) as well as (Y5)-(Y8), then there is $t_0<0$ such that
the geometric $t$-potential $\varphi_t$ satisfies \textbf{(P1)}-\textbf{(P4)} for all $t\in (t_0,1]$.
In particular, there exists a unique equilibrium measure for $\varphi_t$ in the class of lifted measures.
For $t\in (t_0,1)$ it is shown in \cite{PSZ} that this measure has exponential tail.

Therefore, by Theorem \ref{appl: Young ergprop gen}, we have the following.

\begin{theorem}\label{appl: Young ergprop} Let $f:M\to M$ be a $C^{1+\epsilon}$ diffeomorphims of a compact smooth Riemannian manifold $M$ satisfying Conditions (Y1)-(Y3) and (Y5)-(Y8).
	Assume in addition that the inducing scheme $\{S,\tau \}$ satisfies Condition \textbf{(I3)}. Then:

\begin{enumerate}
	\item there exists $t_0<0$ such that for every $t_0<t\leq 1$ there exists a unique equilibrium measure $\mu_t$ among all lifted measures for the potential $\varphi_t.$\\
	
\noindent Assume in addition that $f$ satisfies Condition (Y4). Then:\\
	
	\item the function $P(t):=P_L(\varphi_t)$ is real analytic on the interval $(t_0,1).$\\
	
\noindent Assume in addition that there exists $K>0$ such that for any $i\in\mathbb{N}$ and
	for $x,y\in \Lambda_i^s$ and $0\leq j \leq \tau_i$,
	$$
	d(f^j(x),f^j(y))\leq K \max\left\{ d(x,y), d(\tilde{f}(x),\tilde{f}(y))    \right\}.$$   
	\noindent Then for $t\in(t_0,1)$:
	\item $\mu_t$ has exponential decay of correlations and satisfies the CLT with respect to H\"older continuous observables;
	\item $\mu_t$ has the Bernoulli property. 
\end{enumerate}	

\end{theorem}


\subsection*{Example 1 - the Katok map}
Consider the automorphism of the two-dimensional torus 
$\mathbb{T}^2=\mathbb{R}^2/\mathbb{Z}^2$ given by the matrix 
$T:=\left(\begin{matrix}
2 & 1\\
1 & 1
\end{matrix}\right)$ and then choose a function 
$\psi:[0,1] \to [0,1]$ satisfying: 
\begin{itemize}
\item[(K1)] $\psi$ is of class $C^{\infty}$ except at zero; 

\item[(K2)] $\psi(u)=1$ for $u\geq r_0$ and some $0<r_0<r_1$; 

\item[(K3)] $\psi'(u)>0$ for every $0<u<r_0$; 

\item[(K4)] $\psi(u)=(ur_0)^{\alpha}$ for $0\leq u\leq r_0/2$ where $0<\alpha<1/2$. 
\end{itemize}

Let $D_r=\{ (s_1,s_2) : s_1^2+s_2^2\leq r^2  \}$ where $(s_1,s_2)$ is the coordinate system obtained from the eigendirections of $T$. Choose $r_1 > r_0$ such that 
\begin{equation}
 D_{r_0} \subset  \text{Int}T(D_{r_1})\cap \text{Int}T^{-1}(D_{r_1}) 
 \end{equation}
 and consider the system of differential equations in 
 $D_{r_1}$ 
 \begin{equation}\label{kat-eq}
 \dot{s}_1 = s_1 \log\lambda,~~ \dot{s}_2 = −s_2 \log\lambda,
\end{equation}
where $\lambda > 1$ is the eigenvalue of $T$. Observe that $T$ is the time-one map of the ﬂow generated by the system of equations (\ref{kat-eq}). We slow down trajectories of the system (\ref{kat-eq}) by perturbing it in $D_{r_1}$ as follows
\begin{equation}
\dot{s}_1 = s_1\psi(s_1^2 + s_2^2)\log\lambda, 
~ \dot{s}_2 =−s_2\psi(s_1^2 + s_2^2)\log\lambda. 
\end{equation}

This system of equations generates a local flow. Denote by $g$ the time-one map of this flow. The choices of $\psi$ and $r_0$ and $r_1$ guarantee that the domain of $g$ contains $D_{r_1}$. Furthermore, $g$ is of class $C^{\infty}$ in $D_{r_1}$ except at the origin and it coincides with $T$ in some neighborhood of the boundary $\partial D_{r_1}$. Therefore, the map 
$$G(x) = \left\{ \begin{array}{lcr}
T(x) & \text{if} & x \in \mathbb{T}^2\setminus D_{r_1}, \\ 
g(x) & \text{if}  & x \in D{r_1} 
\end{array} \right. $$

defines a homeomorphism of the torus $\mathbb{T}^2$, which is a $C^{\infty}$ dffeomorphism everywhere except at the origin. The map $G$ preserves a probability measure $\nu$, which is absolutely continuous with respect to the area. The density of $\nu$ is a $C^{\infty}$ function that is infinite at $0$. One can further perturb the map $G$ to obtain an area-preserving $C^{\infty}$ dffeomorphism $f$. This is the Katok map.

The following is shown in \cite[Lemma 6.1., Proposition 6.2.]{THK}.
\begin{proposition}
For $r_1>0$ small enough the map $f$ constructed above is a Young diffeomorphism satisfying Conditions $(Y1)-(Y8)$.
In addition, the corresponding inducing scheme $\{S,\tau \}$ satisfies Condition \textbf{(I3)}.
\end{proposition}

\subsection*{Example 2 - a slow down of a hyperbolic attractor}

Let $M$ be a $d$-dimensional, compact, smooth Riemannian manifold and $U\subset M$ an open set. Let $f: U \to M $ be a $C^{1+\alpha}$ diffeomorphism onto its image
with $\overline{f(U)}\subset U$, where $\alpha \in (0,1).$ \footnote{One may consider $f$ of class $C^{1+r}$ for any $r>0$.
However, in the construction of the map $g$ (see \textbf{(D0)}) one should use a number $0<\alpha<\min\{1,r\}$} Let $\Lambda = \bigcap_{n\geq 0}\overline{f^n(U)}$ be an attractor for $f$ with $NW(f)=\Lambda$. Assume that, 

\textbf{(C1)} $\Lambda$ is a hyperbolic set for $f$,
so that for every $x\in\Lambda$ there exists a splitting of the tangent space,
$T_xM = E^u_{f}(x)\oplus E^s_{f}(x)$, with $Df(x)(E^u_f(x))=E^u_f(f(x))$ and $Df(x)(E^s_f(x))=E^s_f(f(x))$ such that,

\begin{equation}\label{exp}
\begin{array}{ccccc}
\|  Df(x)(v^u)\| &\geq& \nu \| v^u\|  &\text{ for all }& v^u\in E^u_{f}(x),\text{ and}\\
\|  Df(x)(v^s)\| &\leq& \nu^{-1} \| v^s\|  &\text{ for all }& v^s\in E^s_{f}(x),
\end{array}
\end{equation}
for some $\nu>1$.

\textbf{(C2)} The unstable distribution $E^u_f(x)$ is one-dimensional for all $x\in\Lambda$.

\textbf{(C3)} The map $f$ has a fixed point $p\in\Lambda$.

Consider a neighborhood $Z_0\subset U$ of $p$ with local coordinates identifying the decomposition $E^u_f(p)\oplus E^s_f(p)$ with $\mathbb{R}\oplus\mathbb{R}^{d-1}$.

\textbf{(C4)}  There exists a neighborhood $Z\subset Z_0$ of $p$ on which $f$ is the time-$1$ map of the flow generated by
a linear vector field, $ \dot{x}=Ax  $,
where $A=A_u\oplus A_s$ with $A_u=\gamma Id_u$ and $A_s=-\beta Id_s$ for some $\beta>\gamma>0$.

From now on we use local coordinates in $Z$ and identify $p$ with $0$.
Fix $0<r_0<r_1$ such that $B(0,r_1)\subset Z\subset U$, and let $\psi:[0,1] \to [0,1]$ be a $C^{1+\alpha}$ function satisfying the following condition:

\textbf{(D0)} $\begin{array}{lcr}
a) & \psi(r)=r^\alpha  & \text{ for } r\leq r_0; \\
b) & \psi(r)=1  & \text{ for }r\geq r_1;\\
c) & \psi '(r)\geq 0.
\end{array}$

Let $\chi:Z \to \mathbb{R}^d$ be the vector field given by $\chi(x)=\psi(\|x\|)Ax$ and let $g:U \to M$
be the time-1 map of the flow generated by this vector field on $Z$ and by $f$ on $U\setminus Z$.
Observe that $g$ is of class $C^{1+\alpha}$ and that $g(U)=f(U)$, in particular $\overline{g(U)}\subset U$
and then $\Lambda_g:=\bigcap_{n\geq 0}\overline{g^n(U)}$ is an attractor for $g$.

The following is shown in \cite[Theorem 4.2., Lemma 5.13.]{Zel}.

\begin{proposition}
Assume that $f$ is a $C^1$-small perturbation of a certain local diffeomorphism $\bar{f}$, for which the SRB measure $\bar{\omega}_1$ and the measure of maximal entropy $\bar{\omega}_0$ coincide, and let $r_1>0$ be small enough.

Then the map $g$ constructed above is a Young diffeomorphism satisfying Conditions $(Y1)-(Y8)$.
In addition, the corresponding inducing scheme $\{S,\tau \}$ satisfies Condition \textbf{(I3)}.
\end{proposition}

Another example of a system with an inducing scheme of hyperbolic type is the  H\'enon map at the first bifurcation.
The details on the corresponding inducing scheme can be found in \cite{ST} and \cite{ind}.

\section{Proofs}\label{sec:proofs}

\subsection{Preliminaries}\label{sec:prel}

We need to introduce some notations and results.

\subsubsection{Countable Markov shifts.}
In this section we provide some results on thermodynamics for countable Markov shifts. 
Let $S$ be a countable alphabet, $S^{\mathbb Z}$ the space of two-sided sequences and $\sigma$ the left shift $(\sigma(x))_i=x_{i+1}$. For $k\in \mathbb{Z}$, $n\geq 1$ and $a_1,\cdots,a_n$ a \textit{cylinder} set is defined as

$$_k[a_1,\cdots, a_n]=\{\om=(\cdots,\om_{-1},\om_0,\om_1\cdots)\in S^{\mathbb{Z}}: \om_i=a_{i+k-1},\ \ k\le i\le k+n-1\}.$$
To simplify notation we will write $[a_1,\cdots, a_n]$ for $_0[a_1,\cdots, a_n]$.

Let $A=(a_{ij})_{i,j\in\mathbb{Z}}$ be a countable matrix with entries in $\{0,1\}.$ A \textit{two-sided topological Markov shift} (TMS) is a pair $(S_A^{\mathbb{Z}}, \sigma_{A}),$ where 
$S_A^{\mathbb{Z}}=\{\om=(\cdots, \om_{-1},\om_0,\om_1,\cdots): a_{\om_i\om_{i+1}}=1\ \forall i\in\mathbb{N}\},$ and $\sigma_A$ is the restriction of $\sigma$ to $S_A^{\mathbb{Z}}.$ A TMS is \textit{topologically transitive} if for every two states $\om,\om'\in S^{\mathbb{Z}}_A$ there exists $N:=N(\om,\om')\in\mathbb{N}$ and $\xi_1,\cdots \xi_{N-1}$ such that $a_{\om\xi_1}a_{\xi_1\xi_2}\cdots a_{\xi_{N-1}\om'}=1.$ A TMS is \textit{topologically mixing}  if for every two states $\om,\om'\in S^{\mathbb{Z}}_A$ there exists $N:=N(\om,\om')\in\mathbb{N}$ such that for $\forall n\ge N$ there exist $\xi_1,\cdots \xi_{n-1}$ such that $a_{\om\xi_1}a_{\xi_1\xi_2}\cdots a_{\xi_{n-1}\om'}=1.$ Clearly, mixing implies transitivity. In fact, A TMS is topologically mixing if and only if it is topologically transitive and $gcd\{n: \om\overset{n}{\rightarrow}\om \}=1,$ where $\om\overset{n}{\rightarrow}\om$ means there exist $\xi_1,\cdots \xi_{n-1}$ such that $a_{\om\xi_1}a_{\xi_1\xi_2}\cdots a_{\xi_{n-1}\om}=1$ (see \cite{sarigsub} for the argument).

\subsubsection{Functions on TMS}

Given a function $\Phi:S_A^{\mathbb{Z}}\rightarrow \mathbb{R}$ we define the $\textit{n-th\ variation}$ 
of $\Phi$ by

$$Var_n(\Phi)=sup\{|\Phi(\om)-\Phi(\om')|:\om,\om'\in _{-n+1}[a_{1},\cdots, a_{2n-1}]\}.$$

$\Phi$ is said to have a  \textit{summable variation} if
$$\sum\limits_{n\ge 2}Var_n(\Phi)\ <\infty.$$

$\Phi$ is said to be  \textit{locally H\"older continuous} if there exists $C>0$ and $0<r<1$ such that for all $n\ge 1$

$$Var_n(\Phi)\le Cr^n.$$

Let  $\Phi_n=\sum\limits_{k=0}^{n-1}\Phi\circ\sigma^k$ denote the $n$th Birkhoff's sum
of $\Phi.$\\

For $a\in S$ let $\chi_{[a]}$ denotes the characteristic function of the cylinder $[a].$ 
Define $n_a(\om):=\chi_{[a]}(\om) \inf \{ n\geq 1 : \sigma^n(\om)\in [a]   \}$
(where $\inf\emptyset:=\infty$ and $0 \cdot\infty :=0$). Set

$$Z_n(\Phi,a):=\sum\limits_{\sigma^n(\om)=\om}exp(\Phi_n(\om))\chi_{[a]}(\om)$$
and
$$Z^*_n(\Phi,a):=\sum\limits_{\sigma^n(\om)=\om}exp(\Phi_n(\om))\chi_{\{ n_a= n \}}(\om).$$

The \textit{Gurevich pressure} of $\Phi$ is defined by

$$P_G(\Phi)=\lim\limits_{n\to\infty}\frac 1n \log \sum\limits_{\sigma^n(\om)=\om}exp(\Phi_n(\om))\chi_{[a]}(\om)$$

for some $a\in S$.\\

If $(S^{\mathbb{Z}}_A,\sigma_A)$ is topologically mixing and $\Phi$ has a summable variation, then the above limit exists and it is independent of $a\in S$, \cite{Sar99}.\\
Let $\lambda=expP_G(\Phi).$ We say $\Phi$ is\textit{ positive recurrent} if  $\sum\lambda^{-n}Z_n(\Phi,a)=\infty,\ \ \sum n\lambda^{-n}Z^*_n(\Phi,a)<\infty. $

By $Erg(\sigma)$ we denote the set of all $\sigma$- invariant ergodic Borel probability measures and 

$$Erg_{\Phi}(\sigma)=\{\nu\in Erg(\sigma): \int_{S^{\mathbb{Z}}}\Phi d\nu>-\infty\}.$$

A $\sigma$- invariant measure $\nu_{\Phi}$ is an\textit{ equilibrium measure} for $\Phi$ provided

$$P(\Phi):=\underset{\nu\in Erg_{\Phi}(\sigma)}{sup}\{ h_{\nu}(\sigma)+\int\Phi d\nu\} =h_{\nu_{\Phi}}(\sigma)+\int\Phi d\nu_{\Phi}$$

where $h_{\nu}(\sigma)$ denotes the measure-theoretic entropy of $\sigma$ with respect to $\nu.$
We will call the quantity $P(\Phi)$, the \textit{variational pressure} of $\Phi$.

A measure $\nu=\nu_{\Phi}$ is a Gibbs measure for $\Phi$ provided that there exists a  constant $C>0$ such that for any cylinder set $[a_0,\cdots, a_{n-1}]$ and any $\om\in [a_0,\cdots, a_{n-1}]$ one has

$$C^{-1}\le\frac{\nu([a_0,\cdots, a_{n-1}])}{exp(-nP_G(\Phi)+\Phi_n(\om))}\le C.$$

For a $\sigma$- invariant measure $\nu$ we define $\nu\circ\sigma(E)=\sum\limits_{a\in S}\nu(\sigma(E\cap[a])).$ It is easy to verify that $\nu$ is absolutely continuous with respect to $\nu\circ\sigma.$ We define the the \textit{Jacobian} of $\nu$ is defined by $g_{\nu}:=\frac{d\nu}{d\nu\circ\sigma}.$ A non-singular measure which is finite on cylinders is called \textit{conformal} for a potential $\Phi$ if there exists a constant
$\lambda>0$ such that $g_{\nu}=\lambda^{-1}exp\Phi,$ $\nu\circ\sigma-$a.e. \\

The Ruelle's operator for a potential $\Phi$ is given by $L_{\Phi}f=\sum\limits_{\sigma(\omega)=x}e^{\Phi(\omega)}f(\omega).$
Suppose $S^{\mathbb{Z}}_A$ is topologically mixing and $\Phi$ is a positive recurrent potential with finite Gurevich pressure and summable variation.
Then by generalized Ruelle's Perron-Frobenious theorem (Theorem 4.9, \cite{Sar15}) there exist a positive continuous function $h$ and a measure $\nu$ which is finite on cylinders such that $L_{\Phi}h=\lambda h,\ \ L^*_{\Phi}\nu=\lambda\nu \text{and} \int hd\nu=1.$ Here $h$ and $\nu$ are called an eigenfunction and eigenmeasure of Ruelle's operator $L_{\Phi}.$  The measure $hd\nu$ is called Ruelle-Perron-Frobenious(RPF) measure for $\Phi.$ In this case $\lambda=expP_G(\Phi).$ 

\subsubsection{Inducing on Markov shifts}

Fix some state $a\in S$. Set $\bar{S}_A:=\{ [\bar{a}]=[aa_1\ldots a_n] : n\geq 0, a_i\neq a, [\bar{a}a]\neq \emptyset \}$,
and let $\bar{\sigma}: \bar{S}^{\mathbb{Z}}_A \to \bar{S}^{\mathbb{Z}}_A $ be the left shift.
For every $\p: S^{\mathbb{Z}}_A \to \mathbb{R}$ we set
$$ \bar{\p} := \left(   \sum_{k=0}^{n_a -1}  \p\circ \sigma^k   \right)  \circ \bar{\pi},  $$

where $\bar{\pi}:  \bar{S}^{\mathbb{Z}}_A \to [a] $ is the canonical projection.  Namely,
$$\bar{\pi}(\ldots[\bar{a}_{-1}][\bar{a}_0][\bar{a}_1][\bar{a}_2]\ldots ):= (\ldots \bar{a}_{-1}\bar{a}_0\bar{a}_1\ldots).$$

\subsubsection{Symbolic representation of the tower map}
We now construct a particular Markov partition of $\hat{X}$ for the tower map $\hat{f}$ which will be used for our purpose. For $J\in S$ and 
 $0<k\leq \tau(J)-1$ set $\tilde{X}_{J,k}=\{(x,k): x\in J \}. $  Let $\hat{S}=\{\tilde{X}, \tilde{X}_{J,k}\}$.
Then $\hat{f}: \hat{X}\to \hat{X}$ can be represented as a countable Markov shift $\hat{\sigma}_A:\hat{S}^{\mathbb{Z}}_A \to \hat{S}^{\mathbb{Z}}_A$,
where allowed transitions are  $\tilde{X}\rightarrow \tilde{X}_{J,1},\ $  $\tilde{X}_{J,k}\rightarrow \tilde{X}_{J,k+1},\ \ k<\tau(J)-1$ and $\tilde{X}_{J,\tau(J)-1}\rightarrow \tilde{X}.$
 
Similarly as in Section \ref{def: induced} we obtain a coding map $\pi_1: \hat{S}^{\mathbb{Z}}_A \to \bigcup_{J\in S} \bigcup_{k=0}^{\tau(J)-1}\hat{f}^k(\bar{J}) $
which is one-to-one on a certain set $\hat{\Omega}$ with $\pi_1(\hat{\Omega})=\hat{X}$ and by Condition $\textbf{(I3)}$, introduced in Section \ref{sec:Is}, the set $ \hat{S}^{\mathbb{Z}}_A \setminus \hat{\Omega}$ supports no $\hat{\sigma}_A$-invariant measure which gives positive weight to any open set.

The above coding allows us to define regularity properties of functions on $\hat{X}$ in terms of the functions on the TMS $(\hat{S}_A^{\mathbb{Z}},\hat{\sigma}_{A}).$ Namely, we say that a function $\hat{\varphi}$ defined on $\hat{X}$ has property $\alpha$ if the function given by $\hat{\Phi}:=\hat{\varphi}\circ\pi_1$ has property $\alpha$.

\subsection{Proof of Theorem \ref{thm:generalcase}}\label{proof:thm:generalcase}
Using the symbolic representation of $(\tilde{X},\tilde{f})$ and $(\hat{X},\hat{f})$, to show the Bernoulli property of the system $(\hat{X},\hat{f},\hat{\mu})$
it is enough the show the following general statement.

Let $(S_A^{\mathbb{Z}},\sigma)$ be a countable Markov shift. 
Choose a state $s\in S$ and let $(\bar{S}^{\mathbb{Z}},\bar{\sigma})$ be the induced shift on $s\in S$, that is, the first return to $[s]$.
Let $\mu$ be a $\bar{\sigma}$-invariant, ergodic measure on $[s]$.
Denote by $\hat{\mu}$ the measure on $S_A^{\mathbb{Z}}$ which is the lifted measure from $\mu$.
Conditions $\textbf{(B1)}$ and $\textbf{(B2)}$ applied to this setting become the following.

\textbf{(S1)} For every finite $S^*\subset \bar{S}$ there exists a constant $C^*=C^*(S^*)>1$ such that for every
pair of cylinders $[A]=[\bar{a}_1\ldots \bar{a}_{n_1}]$, $[B]=[\bar{b}_1\ldots \bar{b}_{n_2}]$,
if $a_{n_1}\in S^*$ or $b_{1}\in S^*$, and $[AB]\neq \emptyset$, then
\begin{equation}\label{eqn:directproductsym}
\frac{1}{C^*}\leq \frac{\mu[AB]}{\mu[A]\mu[B]}\leq C^*.
\end{equation}

Given two cylinders $A=[\bar{a}_1\ldots \bar{a}_{n_1}]$, $B=[\bar{b}_1\ldots \bar{b}_{n_2}]$, $M\in \mathbb{N}$, and $k\geq M$ we define
$$ S([A],[B],M,k) := \{ x\in    [A] \cap \bar{\sigma}^{-(n_1+M)}[B]  | \tau(\bar{\sigma}^{n_1}x) + \ldots + \tau(\bar{\sigma}^{n_1+M-1}x)=k   \}  .  $$

\textbf{(S2)} For any $\delta>0$ there exists $m_0\in\mathbb{N}$ such that for any $m\geq m_0$ and for any finite collection $\gamma$ of $m-$cylinders
$[C]=[\bar{c}_1 \ldots \bar{c}_m]$,  $\bar{c_i}\in \bar{S}$, there exists $K\in\mathbb{N}$ such that the following statement holds:

Let $C, C'\in \gamma$, $k>K$, and let $[A]=[\bar{a}_1 \ldots \bar{a}_{n_1}]$, $[B]=[\bar{b}_1 \ldots \bar{b}_{n_2}]$, $\bar{a}_i,\bar{b}_i\in \bar{S}$ be any two cylinders. Assume that $[AC], [C'B] \neq \emptyset$. Then
\begin{equation}
\frac{e^{-\delta}}{Q_{\mu}}\leq \frac{\sum_{M=1}^{k} \mu   (  S([AC],[C'B],M,k)  )      }{   \mu([AC])\mu([C'B])   }\leq \frac{e^{\delta}}{Q_{\mu}}.
\end{equation}
We have the following.

\begin{TheoremA}\label{thm:bern.symb0}
If $\mu$ satisfies Conditions \textbf{(S1)} and \textbf{(S2)}, then $\hat{\mu}$ is Bernoulli.
\end{TheoremA}

\begin{proof}
On $[s]$ we define a partition $\bar{S}$ by cylinders of the form $[\bar{a}]=[sa_1a_2\ldots a_{\tau-1}]$, where $\tau=\tau(\bar{a})>1$, $a_i\in S\setminus{s}$, and $[a_{\tau-1}s]\neq\emptyset$.
Consider a partition $\hat{S}$ of $S_A^{\mathbb{Z}}$ by sets of the form $[\bar{a}^k]:=\sigma^k([\bar{a}])$, where $\bar{a}\in \bar{S}$
and $0\leq k\leq \tau(\bar{a})-1$. In other words, if $[\bar{a}]=[sa_1a_2\ldots a_{\tau-1}]$, then $[\bar{a}^k]= _{-k}[sa_1a_2\ldots a_{\tau-1}]=_{-k}[\bar{a}]$.
Observe that $(S_A^{\mathbb{Z}},\sigma)$ can be identified with $(\hat{S}^{\mathbb{Z}}_{\hat{A}},\hat{\sigma})$, where 
the allowed transitions are $\bar{a}^k \to \bar{a}^{k+1}$ for $k<\tau(\bar{a})-1$ and $\bar{a}^{\tau(\bar{a})-1} \to \bar{b}^0$
for $\bar{a}, \bar{b}\in \bar{S}$.

In addition, $\hat{\mu}([\bar{a}^k])=\hat{\mu}([\bar{a}^0 \bar{a}^1\ldots \bar{a}^{\tau(\bar{a})-1}])=\frac{1}{Q_{\mu}} \mu([\bar{a}]).     $

By the argument presented in \cite[proof of Theorem 3.1]{SarBernoulli} it is enough to show the following.

\begin{claim}\label{claimm}
For any finite collection $S^*\subset \bar{S}$
and any $\epsilon>0$ there exists $K>0$ such that if $A:=_{-(\tau(\bar{a}_1) +\ldots+ \tau(\bar{a}_{n_1-1} ) )}[\bar{a}_1\ldots \bar{a}_{n_1}]$
and $B:=_{k-\tau(\bar{b}_1)}[\bar{b}_1\ldots \bar{b}_{n_2}]$ are two nonempty cylinders such that $\bar{b}_1, \bar{a}_{n_1}\in S^*$
and $k\geq K,$ then 
\begin{equation}\label{claim}
 |\hat{\mu}(A\cap B) - \hat{\mu}(A)\hat{\mu}(B) | < \epsilon \hat{\mu}(A)\hat{\mu}(B).   
 \end{equation}
\end{claim}

Claim \ref{claimm} implies that for every finite collection of states, $\tilde{S}\subset \hat{S}$, the corresponding partition, $\{[s] | s\in \tilde{S}  \} \cup \{ \bigcup_{s\in \hat{S}\setminus \tilde{S}} [s]   \}$, is weak Bernoulli \cite{SarBernoulli}.
This in turn implies the Bernoulli property of $\hat{\mu}$ (see \cite{SarBernoulli} for details). \\

We now prove Claim\ref{claimm}.

Fix small $\delta>0$ to be determined later and choose:
\begin{itemize}
\item a constant $C^*=C^*(S^*)$ as in Condition \textbf{(S1)}
\item a natural number $m$ as in Condition \textbf{(S2)}
\item a finite collection $\gamma$ of $m-$cylinders $[C]=[\bar{c}_1\ldots\bar{c}_{m}]$ with $\bar{c}_i\in \bar{S}$
such that $\mu(\bigcup_{C\in\gamma}[C])> 1-\alpha, $ where $\alpha:=\min\{ \frac{\delta}{Q_{\mu}(C^*)^2}, \frac{1-e^{-\delta}}{C^*}  \} $.
\item let $K_1$ be as in Condition \textbf{(S2)}.
\item Define $K_2:=\max_{C\in\gamma}(\tau(\bar{c}_1)+\ldots+\tau(\bar{c}_{m}))$ and $K_3:=\max_{\bar{a}\in S^*}\tau(\bar{a})$.
\item Let $K>K_1+K_2+2 K_3$. 
\end{itemize}

For fixed $C,C'\in\gamma$ and $A,B$ as in the Claim above we denote $\bar{K}:=K -\tau(\bar{b}_1) -\tau(\bar{c}'_1) -\ldots -\tau(\bar{c}'_{m}) -\tau(\bar{c}_1)-\ldots-\tau(\bar{c}_{m})-\tau(\bar{a}_{n_1})$. 

Condition \textbf{(S2)} gives that
\begin{equation}\label{eqn:muhat10}
\begin{aligned}
&\hat{\mu}([AC]\cap\hat{\sigma}^{-(K-\tau(\bar{b}_1)-\tau(\bar{c}'_1)- \ldots- \tau(\bar{c}'_{m}) +\tau(\bar{a}_1) +\ldots + \tau(\bar{a}_{n_1-1})) }[C'B] )=\\
&=\frac{1}{Q_{\mu}} \sum_{M=1}^{\bar{K}} \mu(S([AC],[C'B],M,\bar{K})) \\     
    &=   \frac{e^{\pm \delta}}{Q_{\mu}^2} \mu([AC])\mu([C'B])\\
&=e^{\pm \delta} \hat{\mu}([AC])\hat{\mu}([C'B])
\end{aligned} 
\end{equation}

We now continue with the proof of the Claim. 
We follow the line of the argument presented in \cite{SarBernoulli}.
Define
$$\mathfrak{D}:=\{ (\bar{d}_1,\ldots ,\bar{d}_l) | \bar{d}_i\in\bar{S}, l\geq 1, \tau(\bar{d}_1)+\ldots + \tau(\bar{d}_l)=K-\tau(\bar{a}_{n_1})-\tau(\bar{b}_1)    \}  .  $$
We calculate,
$$ \hat{\mu}(A\cap B) = \sum_{D\in\mathfrak{D}} \hat{\mu} ([ADB])  
  =S_1 + S_2 + S_3 + S_4, \text{ where}$$

$$  S_1:=   \sum_{\substack{D\in\mathfrak{D}, ~ l\geq m\\ [\bar{d}_1,\ldots ,\bar{d}_m]\in\gamma \\  [\bar{d}_{l-m},\ldots ,\bar{d}_l]\in\gamma }} \hat{\mu} ([ADB]); ~~ S_2:=  \sum_{\substack{D\in\mathfrak{D}, ~ l\geq m\\ [\bar{d}_1,\ldots ,\bar{d}_m]\notin\gamma \\  [\bar{d}_{l-m},\ldots ,\bar{d}_l]\in\gamma }} \hat{\mu} ([ADB]) ;$$   
$$ S_3:= \sum_{\substack{D\in\mathfrak{D}, ~ l\geq m \\  [\bar{d}_{l-m},\ldots ,\bar{d}_l]\notin\gamma }} \hat{\mu} ([ADB]); ~~ S_4:= \sum_{\substack{D\in\mathfrak{D}, \\ l\leq m  }} \hat{\mu} ([ADB])  .$$

We start by estimating $S_1$. By (\ref{eqn:muhat10}) we have that 
\begin{align*}
 \sum_{\substack{D\in\mathfrak{D}, ~ l\geq m\\ [\bar{d}_1,\ldots ,\bar{d}_m]\in\gamma \\  [\bar{d}_{l-m},\ldots ,\bar{d}_l]\in\gamma }} \hat{\mu} ([ADB]) &= e^{\pm  \delta} \sum_{C,C'\in \gamma}  \hat{\mu}([AC])\hat{\mu}([C'B]) \\
    &= e^{\pm  \delta} \left(    \sum_{C\in \gamma}  \hat{\mu}([AC]) \right)   \left(    \sum_{C'\in \gamma} \hat{\mu}([C'B])\right). 
    \end{align*}
We estimate the first sum.
\begin{align*}
  \hat{\mu}(A)  &\geq   \sum_{C\in \gamma}  \hat{\mu}([AC])  = \hat{\mu}(A) - \frac{1}{Q_{\mu}}  \sum_{C\notin \gamma}  \mu([AC]) \\
 &\geq \hat{\mu}(A) - \frac{1}{Q_{\mu}} C^* \sum_{C\notin \gamma}  \mu([A])\mu([C]) \\
 &= \hat{\mu}(A) - \hat{\mu}(A) C^* \sum_{C\notin \gamma} \mu([C]) \\
 &\geq  \hat{\mu}(A) (1-C^*\alpha)\geq  \hat{\mu}(A) e^{-\delta} .  
 \end{align*}
 
 Similarly we obtain that  $$ \hat{\mu}(B)  \geq \sum_{C'\in \gamma} \hat{\mu}([C'B])\geq  \hat{\mu}(B) e^{-\delta} .$$
 Consequently, $$ S_1=e^{\pm 3 \delta} \hat{\mu}(A)\hat{\mu}(B).$$

We now proceed with the estimate of $S_2$. Using Condition \textbf{(S1)}, we obtain that
\begin{align*}
 S_2&= \sum_{\substack{D\in\mathfrak{D}, ~ l\geq m\\ [\bar{d}_1,\ldots ,\bar{d}_m]\notin\gamma \\  [\bar{d}_{l-m},\ldots ,\bar{d}_l]\in\gamma }} \hat{\mu} ([ADB]) 
 =  \sum_{\substack{D\in\mathfrak{D}, ~ l\geq m\\ [\bar{d}_1,\ldots ,\bar{d}_m]\notin\gamma \\  [\bar{d}_{l-m},\ldots ,\bar{d}_l]\in\gamma }} \frac{1}{Q_{\mu}}\mu ([ADB])\\
 &\leq  \sum_{\substack{D\in\mathfrak{D}, ~ l\geq m\\ [\bar{d}_1,\ldots ,\bar{d}_m]\notin\gamma \\  [\bar{d}_{l-m},\ldots ,\bar{d}_l]\in\gamma }} \frac{1}{Q_{\mu}}  (C^*)^2\mu ([A])\mu([D])\mu([B])\\
  &=   Q_{\mu}  (C^*)^2\hat{\mu} ([A])\hat{\mu}([B]) \sum_{\substack{D\in\mathfrak{D}, ~ l\geq m\\ [\bar{d}_1,\ldots ,\bar{d}_m]\notin\gamma \\  [\bar{d}_{l-m},\ldots ,\bar{d}_l]\in\gamma }} \mu([D])\\
 &\leq   Q_{\mu}  (C^*)^2\hat{\mu} ([A])\hat{\mu}([B]) \sum_{C=(\bar{c}_1\ldots\bar{c}_m)\notin\gamma }\mu([C])\\
 &\leq   Q_{\mu}  (C^*)^2  \alpha \hat{\mu} ([A])\hat{\mu}([B]) \leq \delta\hat{\mu} ([A])\hat{\mu}([B]).
 \end{align*}

Similarly we obtain that 
 $$S_3< \delta \hat{\mu} ([A])\hat{\mu}([B]).$$
It remains to estimate $S_4$.
For this observe that any cylinder $D=[\bar{d_1}\ldots\bar{d}_l]\in\mathfrak{D}$ with $l\leq m$
can be written as the following disjoint union 
$$ \bigcup_{(\bar{d'}_{l+1}\ldots\bar{d'}_{m})\in \bar{S}^{m-l}} [   \bar{d_1}\ldots\bar{d}_l \bar{d'}_{l+1}\ldots\bar{d'}_{m} ]. $$  
Each cylinder $[\bar{d_1}\ldots\bar{d}_l \bar{d'}_{l+1}\ldots\bar{d'}_{m} ]$ in the union above is of length exactly $m$. 
Clearly,
$[\bar{d_1}\ldots\bar{d}_l \bar{d'}_{l+1}\ldots\bar{d'}_{m} ] \notin \gamma$ because $\tau(\bar{d}_1)+\ldots + \tau(\bar{d}_l)=K-\tau(\bar{a}_{n_1})-\tau(\bar{b}_1) > K_2.$ 
We therefore estimate $S_4$ as follows,
\begin{align*}
 S_4&= \sum_{\substack{D\in\mathfrak{D}, \\ l\leq m  }} \hat{\mu} ([ADB]) 
\leq  Q_{\mu}  (C^*)^2\hat{\mu} ([A])\hat{\mu}([B])  \sum_{\substack{D\in\mathfrak{D}, \\ l\leq m  }}  \mu([D])\\
&=  Q_{\mu}  (C^*)^2\hat{\mu} ([A])\hat{\mu}([B])  \sum_{\substack{D=(\bar{d_1}\ldots\bar{d}_l)\in\mathfrak{D}, \\ l\leq m  }} 
\sum_{(\bar{d'}_{l+1}\ldots\bar{d'}_{m})\in \bar{S}^{m-l}} \mu([   \bar{d_1}\ldots\bar{d}_l \bar{d'}_{l+1}\ldots\bar{d'}_{m} ])\\
&\leq  Q_{\mu}  (C^*)^2\hat{\mu} ([A])\hat{\mu}([B])\sum_{C=(\bar{c}_1\ldots\bar{c}_m)\notin\gamma }\mu([C]) \\ 
&\leq   Q_{\mu}  (C^*)^2  \alpha \hat{\mu} ([A])\hat{\mu}([B]) \leq \delta \hat{\mu} ([A])\hat{\mu}([B])  .
\end{align*}
We conclude that 
$$ 
|\hat{\mu} (A\cap B) - \hat{\mu} (A)\hat{\mu}(B) | < (1-e^{-3\delta} + 3\delta) \hat{\mu} (A)\hat{\mu}(B).  
$$
To finish the proof of the Claim, choose $\delta < \min \{\frac{\epsilon}{6} , \log \left(  \frac{2-\epsilon}{2}   \right)^{-\frac{1}{3}}  \}.$\\
\end{proof}
To prove that $(X,f,\mu)$ is Bernoulli, we define an equivalence relation $\sim$ on $\hat{X}$ in the following way. Having two elements, $(x,k), (y,l)\in \hat{X}$
we will say that $(x,k)\sim (y,l)$ if and only if $f^k(x)=f^l(y)$. Consider a $\sigma-$algebra $\mathcal{A}$ on $\hat{X}$ generated by sets of the form
$$ A= \{  (x,k) |\text{ there is }(y,l)\in B \text{ such that } (x,k)\sim (y,l)   \} , $$
where $B\subset \hat{X}$ is a Borel set. Observe that $\mathcal{A}$ is $\hat{f}$-invariant. In addition, as a factor of $(\hat{X},\hat{f},\hat{\mu})$, the restriction $(\hat{X},\hat{f},\hat{\mu}, \mathcal{A})$ is Bernoulli \cite{Orn}.

On the other hand, $\hat{f}$ restricted to $\mathcal{A}$ can be thought of as a map defined on the space $\hat{Y}$ of equivalence classes
$\hat{Y}= \{ [\hat{x}]_{\sim} | \hat{x}\in \hat{X}   \}$, where we define $\hat{f}([\hat{x}]_{\sim} ) = [\hat{f}(\hat{x})]_{\sim} . $
We can define an isomorphism $I$ between $(X,f,\mu)$ and  $(\hat{Y},\hat{f},\hat{\mu})$ in the following way.
Given $p\in X$ let $k$ be the smallest positive integer such that $p=f^k(q)$ for some $q\in \tilde{X}$. Then define $I(p):= [q,k]_{\sim}$.
We conclude that $(X,f,\mu)$ and  $(\hat{Y},\hat{f},\hat{\mu})$ are isomorphic.
Therefore, $(X,f,\mu)$ is Bernoulli.

\subsection{Proof of Theorem \ref{thm: tower}}\label{proof:thm:tower}

We start with the proof of Statement (1).
Viewing $\tilde{f}:\tilde{X} \to \tilde{X}$ as an induced system from $\hat{f}$ and using Statement (1) of Theorem \ref{thm: manifold}, it is enough to show that
$P_L(\varphi) = P_L(\hat{\varphi})$. 
 By Proposition 4.3 in \cite{ind}, 
 for every $\tilde{\mu}\in Erg(\tilde{X},\tilde{f})$ for which $Q_{\tilde{\mu}}<\infty$ we have that
 $$ h_{\tilde{\mu}}(\tilde{f})= Q_{\tilde{\mu}} h_{\mathcal{L}(\tilde{\mu})}(f) \text{ and } \int_{\tilde{X}}\tilde{\varphi}d\tilde{\mu} =    Q_{\tilde{\mu}} \int_X \varphi d\mathcal{L}(\tilde{\mu}).  $$
 
 On the other hand, viewing $\tilde{f}:\tilde{X} \to \tilde{X}$ as an induced system from $\hat{f}$ and observing that $\tilde{\hat{\varphi}}=\tilde{\varphi}$
 we obtain that,
$$ h_{\tilde{\mu}}(\tilde{f})= Q_{\tilde{\mu}} ~h_{\hat{\mathcal{L}}(\tilde{\mu})}(\hat{f}) \text{ and }  \int_{\tilde{X}}\tilde{\varphi}d\tilde{\mu}=    Q_{\tilde{\mu}} \int_{\hat{X}} \hat{\varphi} d{\hat{\mathcal{L}}(\tilde{\mu})}.  $$
We conclude that $$ h_{\hat{\mathcal{L}}(\tilde{\mu})}(\hat{f}) = h_{\mathcal{L}(\tilde{\mu})}(f) \text{ and }  \int_{\hat{X}} \hat{\varphi} d\hat{\mathcal{L}}(\tilde{\mu}) =  \int_{M} \varphi d \mathcal{L}(\tilde{\mu}).  $$

Consequently, $P_L(\varphi) = P_L(\hat{\varphi})$. \\

Now we prove Statement (2). In fact, we show a more general result in a setting of symbolic dynamical systems.

Let $(S_A^{\mathbb{Z}},\sigma)$ be a topologically mixing countable Markov shift. 
Choose a state $s\in S$ and a potential $\varphi : [s] \to \mathbb{R}$. Assume that $\tilde{\varphi}$ is locally H\"older with respect to the induced shift $\bar{\sigma}$ on $[s]$, and that $P_G(\varphi)<\infty$.
Let $\mu$ be the unique ergodic equilibrium measure for $\varphi$. 
Denote by $\hat{\mu}$ the measure on $S_A^{\mathbb{Z}}$ which is the lifted measure from $\mu$.
We have the following.

\begin{TheoremB}\label{thm:bern.symb}
If $(S_A^{\mathbb{Z}},\sigma, \hat{\mu})$ is mixing, then it is Bernoulli.
\end{TheoremB}

Our argument follows the idea presented by Sarig in \cite{SarBernoulli}. The result in
\cite{SarBernoulli} considers equilibrium measures for locally H\"older potential functions.
We remark that our result is stronger, as it only requires that the induced potential is locally H\"older.

\begin{proof}
In view of Theorem \ref{thm:generalcase}, it is enough to verify that $\mu$ satisfies Conditions \textbf{(S1)} and \textbf{(S2)}
in Section \ref{proof:thm:generalcase}.
As in the proof of Theorem \ref{thm:generalcase}, we identify  $(S_A^{\mathbb{Z}},\sigma)$ with $(\hat{S}^{\mathbb{Z}}_{\hat{A}},\hat{\sigma})$.

The properties of $\mu$ have been described by Sarig in \cite{SarBernoulli}. It is done by first studying the space of one-sided sequences $(\bar{S}^{\mathbb{N}},\bar{\sigma})$. For any fixed $\phi :\bar{S}^{\mathbb{N}} \to \mathbb{R} $ one defines a Ruelle's operator $L_{\phi}$ in the following way. If $F :\bar{S}^{\mathbb{N}} \to \mathbb{R} $ and $x^+\in \bar{S}^{\mathbb{N}} $, then
$$ L_{\phi} F (x^+) :=  \sum_{\bar{\sigma}(y^+)=x^+} e^{\phi(y^+)} F(y^+). $$

The following is shown in \cite{SarBernoulli}.

\begin{lemma}\label{lem:Sarig}
There exists a locally H\"older continuous function $\phi:\bar{S}^{\mathbb{N}} \to \mathbb{R}$ and the corresponding Ruelle's operator
$L  := L_{\phi} $ such that:
\begin{enumerate}
\item for any cylinder $[A]=[\bar{a}_1\ldots \bar{a}_{n}]$ with $ \bar{a}_i\in \bar{S}$ the measure
$\mu([A])$ is the pointwise limit of the sequence $L^n\chi_{[A]}$;
\item For any two functions $F,G:  \bar{S}^{\mathbb{N}} \to \mathbb{R} $ 
\begin{equation}\label{eq:sar1}
\int F(G\circ \bar{\sigma}^n)d\mu = \int (L^nF)Gd\mu;
\end{equation}
\item For every finite $S^*\subset \bar{S}$ there exists a constant $C^*=C^*(S^*)>1$ such that for every
pair of cylinders $[A]=[\bar{a}_1\ldots \bar{a}_{n}]$, $[B]=[\bar{b}_1\ldots \bar{b}_{m}]$,
\begin{enumerate}
\item if the last symbol in $A$ is in $S^*$ and $[AB]\neq \emptyset$, then $\frac{1}{C^*}\leq \frac{\mu[AB]}{\mu[A]\mu[B]}\leq C^*$;
\item if the first symbol in $A$ is in $S^*$ and $[BA]\neq \emptyset$, then $\frac{1}{C^*}\leq \frac{\mu[BA]}{\mu[B]\mu[A]}\leq C^*$.
\end{enumerate}
\end{enumerate}
\end{lemma}

Observe that Statement (3) of Lemma \ref{lem:Sarig} implies Condition \textbf{(S1)}. We now show that $\mu$ satisfies \textbf{(S2)}.
For that we choose a natural number $m_0$ such that $\sup_{n\geq 1}(Var_{n+m_0}\phi_n) \leq \delta/3$.
Fix $m\geq m_0$ and take any finite collection $\gamma$ of $m-$cylinders $[C]=[\bar{c}_1\ldots\bar{c}_{m}]$ with $\bar{c}_i\in \bar{S}$.
Let $K\in\mathbb{N}$ be such that for any pair of cylinders $[C], [C']\in \gamma$ and any $k>K$ one has
$\hat{\mu}([C]\circ \hat{\sigma}^{-k}([C'])) = e^{\pm\delta/3}\hat{\mu}([C])\hat{\mu}([C'])$.

We start with some technical results.

\begin{lemma}\label{lem:ber1}
Let $n\geq 1$ and $l\geq m$. Choose any $n-$cylinder $A=[\bar{a}_1\ldots \bar{a}_{n}]$, with $\bar{a}_i\in \bar{S}$ and any $l-$cylinder
$C=[\bar{c}_1\ldots \bar{c}_{l}]$, with $\bar{c}_i\in \bar{S}$.
For any one-sided $x\in [AC]$ one has
\begin{equation}\label{Sn}
e^{\phi_n(x)}= e^{\pm\delta/3}\frac{\mu([AC])}{\mu([C])}.
\end{equation}
\end{lemma}

\begin{proof}
Fix $x\in[AC]$ and $p\in \bar{S}^{\mathbb{N}}$. We have that
\begin{align*}
\mu([AC])&=\lim_{N\to \infty} L^N \chi_{[AC]}(p) = \lim_{N\to \infty} \sum_{\bar{\sigma}^N(y)=p} e^{\phi_N(y)}\chi_{[AC]}(y)\\
  &= \lim_{N\to \infty}  e^{\pm\delta/3}e^{\phi_n(x)} \sum_{\bar{\sigma}^{N-n}(y)=p} e^{\phi_{N-n}(y)}\chi_{[C]}(y)\\
&= \lim_{N\to \infty}  e^{\pm\delta/3}e^{\phi_n(x)} L^{N-n}\chi_{[C]} (p)\\
 &= e^{\pm\delta/3}e^{\phi_n(x)} \mu([C]). 
 \end{align*}
\end{proof}

\begin{lemma}\label{lem:ber2}
If $C,C'\in\gamma$ and $p\in [C']$ is a one-sided sequence, then for any $k>K$ we have the following 

\begin{equation}\label{estimate:Ln3}
 \sum_{M=1}^{k} L^{M+m}\left( \chi_{[C]}  \chi_{\{\tau\circ\bar{\sigma}^{m}+\ldots+\tau\circ\bar{\sigma}^{M+m-1}=k\}}\right) (p) = e^{\pm 2\delta/3} \mu([C]) \frac{1}{Q_{\mu}}. 
\end{equation}
\end{lemma}

\begin{proof}
First recall that by the choice of $k,$ 
$$\mu([C]\circ \hat{\sigma}^{-k}([C']))=Q_{\mu}\hat{\mu}([C]\circ \hat{\sigma}^{-k}([C'])) = e^{\pm\delta/3}\hat{\mu}([C])\hat{\mu}([C'])
=e^{\pm \delta/3}\frac{1}{Q_{\mu}}\mu([C])\mu([C']).$$

On the other hand,
\begin{align*}
\mu([C]\circ \hat{\sigma}^{-k}([C']))&= \sum_{M=1}^{k}\sum_{\substack{D=(\bar{d}_1\ldots\bar{d}_{M})\in \bar{S}^M,\\  \tau(\bar{d}_1)+\ldots+\tau(\bar{d}_{M})=k}}   \mu([CDC'])\\
&=\sum_{M=1}^{k}\sum_{\substack{D=(\bar{d}_1\ldots\bar{d}_{M})\in \bar{S}^M,\\  \tau(\bar{d}_1)+\ldots+\tau(\bar{d}_{M})=k}}   \lim_{N\to\infty}L^N\chi_{[CDC']}(x)\text{ for some }x\in\bar{S}^{\mathbb{N}}.
\end{align*}
    
For any $M-$cylinder $D$ in the sum above we can define a unique point $y_D\in [CDC']$ such that $\bar{\sigma}^{m+M}(y_D)=p$.
We then continue,
\begin{align*}
&=\sum_{M=1}^{k}\sum_{\substack{D=(\bar{d}_1\ldots\bar{d}_{M})\in \bar{S}^M,\\  \tau(\bar{d}_1)+\ldots+\tau(\bar{d}_{M})=k}}   e^{\pm\delta/3}e^{\phi_{m+M}(y_D)}\lim_{N\to\infty}L^{N-(m+M)}\chi_{[C']}(x)\\
&=    e^{\pm\delta/3}  \sum_{M=1}^{k}   \sum_{\bar{\sigma}^{m+M}(y)=p}  e^{\phi_{m+M}(y)} \chi_{[C]}(y)  \chi_{\{\tau\circ\bar{\sigma}^{m}+\ldots+\tau\circ\bar{\sigma}^{M+m-1}=k\}}(y)    \mu([C'])\\
&=    e^{\pm\delta/3}  \sum_{M=1}^{k} L^{M+m}\left( \chi_{[C]}  \chi_{\{\tau\circ\bar{\sigma}^{m}+\ldots+\tau\circ\bar{\sigma}^{M+m-1}=k\}}\right) (p)  \mu([C']).
\end{align*}
\end{proof}

We now use Lemma \ref{lem:ber1} and Lemma \ref{lem:ber2} to prove the following.

\begin{lemma}\label{lem:ber3}
If $C,C'\in\gamma$ and $p\in [C']$ is a one-sided sequence, then for any $k>K$ and   
any $A=[\bar{a}_1\ldots \bar{a}_{n}]$, with $n\geq 1,~\bar{a}_i\in \bar{S}$
 we have the following 

$$\sum_{M=1}^{k}     L^{n+m+M}\left(\chi_{[AC]}  \chi_{\{\tau\circ\bar{\sigma}^{n+m}+\ldots+\tau\circ\bar{\sigma}^{n+m+M-1}=k\}} \right)(p)
=  e^{\pm \delta} \mu([AC]) \frac{1}{Q_{\mu}}.$$

\end{lemma}

\begin{proof}

Fix a one-sided sequence $p\in [C']$. We have that
\begin{align*}
&\sum_{M=1}^{k}     L^{n+m+M}\left(\chi_{[AC]}  \chi_{\{\tau\circ\bar{\sigma}^{n+m}+\ldots+\tau\circ\bar{\sigma}^{n+m+M-1}=k\}} \right)(p)\\
=& \sum_{M=1}^{k} \sum_{\bar{\sigma}^{n+m+M}(y)=p} e^{\phi_{n+m+M}(y)}\chi_{[AC]}(y)  \chi_{\{\tau\circ\bar{\sigma}^{n+m}+\ldots+\tau\circ\bar{\sigma}^{n+m+M-1}=k\}}(y). 
\end{align*}
By Lemma \ref{lem:ber1}, this is equal to
\begin{align*}
& e^{\pm\delta/3} \frac{\mu([AC])}{\mu([C])} \sum_{M=1}^{k} \sum_{\bar{\sigma}^{m+M}(y)=p} e^{\phi_{m+M}(y)}\chi_{[C]}(y)  \chi_{\{\tau\circ\bar{\sigma}^{m}+\ldots+\tau\circ\bar{\sigma}^{m+M-1}=k\}}(y) \\
=&  e^{\pm\delta/3} \frac{\mu([AC])}{\mu([C])} \sum_{M=1}^{k} L^{m+M}\left( \chi_{[C]}  \chi_{\{\tau\circ\bar{\sigma}^{m}+\ldots+\tau\circ\bar{\sigma}^{m+M-1}=k \}}\right) (p).
\end{align*}
By Lemma \ref{lem:ber2}, we conclude that
$$\sum_{M=1}^{k}     L^{n+m+M}\left(\chi_{[AC]}  \chi_{\{\tau\circ\bar{\sigma}^{n+m}+\ldots+\tau\circ\bar{\sigma}^{n+m+M-1}=k\}} \right)(p)
=  e^{\pm \delta} \mu([AC]) \frac{1}{Q_{\mu}}.$$
\end{proof}

We continue with the proof of Condition \textbf{(S2)}.
Fix $C,C'\in\gamma$, $k>K$, and let $A,B$  be as in the hypothesis of Condition \textbf{(S2)}. We then calculate,

\begin{align*}
  &\sum_{M=1}^{k} \mu   (   \{ x\in    [AC] \circ \bar{\sigma}^{-(n_1+m+M)}[C'B]  | \tau(\bar{\sigma}^{m+n_1}x) + \ldots + \tau(\bar{\sigma}^{m+n_1+M-1}x)=k   \}   )         \\
&=  \sum_{M=1}^{k} \int \chi_{[AC]}(\chi_{[C'B]}\circ\bar{\sigma}^{n_1+m+M}) \chi_{\{\tau(\bar{\sigma}^{n_1+m}(x))+\ldots+\tau(\bar{\sigma}^{n_1+m+M-1}(x))=k\}}d\mu \text{ by (\ref{eq:sar1}),}\\
&=\sum_{M=1}^{k} \int_{[C'B]} L^{n_1+m+M}\left(\chi_{[AC]}  \chi_{\{\tau\circ\bar{\sigma}^{n_1+m}+\ldots+\tau\circ\bar{\sigma}^{n_1+m+M-1}=k\}} \right)d\mu.
\end{align*}

By Lemma \ref{lem:ber3}, we conclude that

\begin{align*}
 &\sum_{M=1}^{k} \mu   (   \{ x\in    [AC] \circ \bar{\sigma}^{-(n_1+m+M)}[C'B]  | \tau(\bar{\sigma}^{m+n_1}x) + \ldots + \tau(\bar{\sigma}^{m+n_1+M-1}x)=k   \}   ) \\
  &=e^{\pm \delta} \hat{\mu}([AC])\hat{\mu}([C'B])  \frac{1}{Q_{\mu}}.
  \end{align*}

The remaining statements of Theorem \ref{thm: tower} can be deduced from the following, more general result.

\begin{proposition}\label{noequilibrium}
	Let $\tilde{f}:\tilde{X}\to \tilde{X}$ be an induced map satisfying \textbf{(I1)-(I2), (I5)} and $\tilde{\mu}$ be some $\tilde{f}$-invariant ergodic probability measure. 
	Assume in addition that there exist $C>0$ and $0<\rho<1$ such that
	\begin{equation}\label{holderdistorsion}
	|\frac{g_{\tilde{\mu}}(x)}{g_{\tilde{\mu}}(y)}-1|\le C\rho^{s(\tilde{f}(x),\tilde{f}(y))}
	\end{equation}
where 
$g_{\tilde{\mu}}$ is the Jacobian of $\tilde{\mu}$ defined as in Section \ref{sec:prel}. Let $\mu$ and $\hat{\mu}$ denote the lifted measures of $\tilde{\mu}$ to $X$ and $\hat{X}$ respectively. Let $\hat{h}_1,\ \hat{h}_2\in C^{\alpha}(\hat{X})$ and $h_1, h_2\in C^{\alpha}(X).$  Then
	\begin{enumerate}
		\item Assume that $\tilde{\mu}$ has exponential (respectively, polynomial) tail. Then $\hat{f}$ has exponential (respectively, polynomial) decay of correlations with respect to $\hat{\mu};$ 
		\item Assume $\tilde{\mu}(\tau>n)=\mathcal{O}(\frac{1}{n^{\beta}}),\ \beta>1.$ If $\hat{h}_1,\ \hat{h}_2 $ supported on $\bigcup_{j=0}^{k}\tilde{X}\times\{j\}$ ($\tilde{X}=\tilde{X}\times\{0\}$) for some $k, $ then $\hat{f}$ satisfies (\ref{Gouzel}) as well as the CLT with respect to $\hat{\mu};$ 
		\item Assume that $\tilde{\mu}$ has exponential (respectively, polynomial) tail, and $f$ satisfies \textbf{(F1)-(F2)}. Then $f$ has exponential (respectively, polynomial) decay of correlations with respect to $\mu;$
		
		\item Assume $\tilde{\mu}(\tau>n)=\mathcal{O}(\frac{1}{n^{\beta}}),\ \beta>1$ and  $f$ satisfies \textbf{(F1)-(F2)}. There exists a sequence of nested sets $Y_0\subset Y_1\subset\cdots$ in $X$ such that if $h_1,\ h_2$ are supported inside $Y_k$ for some $k\ge 0,$    then $f$ satisfies (\ref{Gouzel1}) as well as the CLT with respect to $\mu.$

	\end{enumerate}
\end{proposition}

The above proposition follows from results in \cite{you2} and \cite{mel}.
We note that our choice of observables ensures that the results in \cite{mel} apply as piecewise H\"older continuous functions are excellent in the sense defined in \cite{mel} (one can see this on the page 896 in \cite{mel}).

To show correlation estimates (3a) and (3b) we apply Statement (1) of Proposition \ref{noequilibrium}.
We need to verify that $\nu_{\varphi^{+}}$  satisfies (\ref{holderdistorsion}). This can be done in the following way. 
Since $\nu_{\varphi^{+}}$ is the unique equilibrium measure then by Theorem 5.5 in \cite{Sar15},  it is equal to the Ruelle-Perron-Frobenius measure for $\varphi^{+}.$ If $\nu_c$ is an eigenmeasure for the Ruelle operator $L_{\varphi^{+}},$ then $d\nu_{\varphi^{+}}=hd\nu_c$ where $h$ is an eigenfunction for $L_{\varphi^{+}}.$ Note that by Theorem 3.3 in \cite{Sar15}, the measure $\nu_c$ is conformal for $\varphi^{+}.$ In addition, since $\varphi$ has zero Gurevich pressure and is positive recurrent (the latter follows from Theorem 5.5 \cite{Sar15}), then by Theorem 3.7 in \cite{Sar15}, it follows that $\varphi^{+}=\log( g_{\nu_c}).$ One can see that
$$
\log(g_{\nu_{\varphi^{+}}})=\log(g_{\nu_c})+\log(h)-\log(h\circ\hat{f}).
$$
By Proposition 3.4 in \cite{SarNotes}, one has 
$Var_nh<\sum\limits_{l\ge n+1}Var_l(\varphi^{+}).$ Hence, by local H\"older continuity of 
$\varphi^{+}$, all the terms in the right hand side of the above equality are locally H\"older continuous with the same exponent. Hence, we find that $\log(g_{\nu_{\varphi^{+}}})$ is locally H\"older continuous function and therefore the measure $\nu_{\varphi^{+}}$ satisfies (\ref{holderdistorsion}) .\\

Note that $\int_{\tilde{X}}\tau(x)dx<\infty$ as $\nu_{\varphi^{+}}$ has decaying tail. In addition, the tower satisfies $\textbf{(I5)}.$ Therefore all the requirements of Proposition \ref{noequilibrium} are met. Hence we obtain correlation estimates (3a)-(3d).
\end{proof}

\subsection{Proof of Theorem \ref{thm: manifold}}\label{proof:thm:manifold}
Statement (1) is proved in \cite{ind}. Statement (2) follows from Statement (1) of Theorem \ref{thm: tower} and Statement (1) of Theorem \ref{thm: manifold}. Statement (3) follows from Theorem \ref{thm:generalcase} and the proof of Theorem \ref{proof:thm:tower}.

We showed in the proof of Theorem \ref{thm: tower} the measure $\nu_{\varphi^{+}}$ satisfies all the assumptions of Proposition \ref{noequilibrium}. Therefore, Statements (4a) and (4b) follow from Proposition \ref{noequilibrium}.

To prove (4c) and (4d)   assume that $\nu_{\varphi^{+}}(\tau>n)=\mathcal{O}(\frac 1{n^{\beta}}),\ \beta>1.$ Define $Y_k=\bigcup_{j=0}^kf^j(\tilde{X}).$ Assume that $h_1, h_2$ supported on $Y_k$ for some $k\ge 0.$ Then the observables $\hat{h}_1, \hat{h}_2$ are supported in  $\bigcup_{j=0}^k\tilde{X}\times \{j\}$ for some $k.$  Again, Statements (4c) and (4d)  follow from Proposition \ref{noequilibrium}.

\subsection{Proof of Theorem \ref{thm:pressure}}\label{proof:thm:pressure}

\subsubsection{Phase transitions for 2-sided countable Markov shifts}

We start by showing an analogous result in symbolic settings. Our argument closely follows ideas presented by Sarig in \cite{Sar00}, and extends the result of Sarig to the case of hyperbolic towers. All necessary definitions and notation are given in Section \ref{sec:prel}. 
Let $(S_A^{\mathbb{Z}},\sigma_A)$ be a topologically mixing countable Markov shift
and let $\p_1,\p_2:S_A^{\mathbb{Z}} \to \mathbb{R}$ be some functions. 

\begin{theorem}\label{presshift}
Assume there exists a state $a\in S$ such that $\p_1,\p_2$ and the corresponding induced potentials $\overline{\p}_1,\overline{\p}_2$ on $[a]$ satisfy the following conditions:
\begin{enumerate}
\item $P(\p_1) < \infty$, $P(\p_1 + t\p_2)<\infty$, and $P_G(\overline{\p_1+t\p_2-P(\p_1+t\p_2)})=0$ for small enough $|t|$;
\item $\overline{\p}_1,\overline{\p}_2$ are locally H\"{o}lder; 
\item $\exists_{M\in\mathbb{R}}~ \sum_{\bar{a}\in\bar{S}_A} \exp\left[  \overline{\p}_1(\ldots\bar{x}_{-2}\bar{x}_{-1}\bar{a}\bar{x}_1\bar{x_2}\ldots) - P(\p_1)|\bar{a}|       \right] <M$ for all $\left(\bar{x}_i \right)_{i=-\infty}^{\infty} \in \bar{S}_A^{\mathbb{Z}}$;
\item there exists $\bar{a}\in\bar{S}_A$ such that:
\begin{enumerate}
\item $\sum_{n\geq 1} Z_n \left( \overline{\p_1-P(\p_1)}, \bar{a}  \right) = \infty$;
\item $\sum_{n\geq 1} n Z_n^* \left( \overline{\p_1-P(\p_1)}, \bar{a}  \right) < \infty$;
\item $\exists_{r>exp(-P(\p_1)), \epsilon_0>0}  \sum_{n\geq 1} n r^n Z_n^*(\p_1 + \epsilon_0 |\p_2|, a) <\infty$.
\end{enumerate}
\end{enumerate}
Then for some $0<\epsilon<\epsilon_0$ the functions $t \to P_G(\overline{\p_1 - P(\p_1)+ t\p_2})$ and $t \to P(\p_1 + t\p_2)$ are real analytic on $(-\epsilon, \epsilon)$.
\end{theorem}

Here $P_G$ denotes the Gurevich pressure and $P$ denotes the variational pressure (see Section \ref{sec:prel} for definitions).

The proof requires two lemmas that are stated below. Consider potentials 
$\p_1^*,\p_2^*: S_A^{\mathbb{N}} \to \mathbb{R}$ on a topologically mixing one-sided Markov shift, a state $a\in S$ and the induced system 
$\left( \bar{S}_A^{\mathbb{N}},\bar{\sigma}\right)$ on $[a]$.

\begin{lemma}\label{le1}
Assume the following conditions are satisfied:
\begin{enumerate}
\item $P(\p_1^*) < \infty$, $P(\p_1^* + t\p_2^*)<\infty$, and $P_G(\overline{\p_1^*+t\p_2^*-P(\p_1^*+t\p_2^*)})=0$ for small enough $|t|$;
\item $\bar{\p}_1^*,\bar{\p}_2^*$ are locally H\"{o}lder;
\item $\exists_{M\in\mathbb{R}}~ \sum_{\bar{a}\in\bar{S}_A} \exp\left[  \bar{\p}_1^*(\bar{a}\bar{x}_1\bar{x_2}\ldots) - P(\p_1^*)|\bar{a}|       \right] <M$ for all $\left(\bar{x}_i \right)_{i=0}^{\infty} \in \bar{S}_A^{\mathbb{N}}$;
\item there exists $\bar{a}\in\bar{S}_A$ such that
\begin{enumerate}
\item $\sum_{n\geq 1} Z_n \left( \overline{\p_1^*-P(\p_1^*)}, \bar{a}  \right) = \infty$;
\item  $\sum_{n\geq 1} n Z_n^* \left( \overline{\p_1^*-P(\p_1^*)}, \bar{a}  \right) < \infty$;
\item $\exists_{r>exp(-P(\p_1^*)), \epsilon_0>0}  \sum_{n\geq 1} n r^n Z_n^*(\p_1^* + \epsilon_0 |\p_2^*|, a) <\infty$.
\end{enumerate}
\end{enumerate}
Then for some $0<\epsilon<\epsilon_0$ the functions $t \to P_G(\overline{\p_1^* - P(\p_1^*)+ t\p_2^*})$ and $t \to P(\p_1^* + t\p_2^*)$ are real analytic on $(-\epsilon, \epsilon)$.
\end{lemma}

\begin{proof}
By (1) we have that $P(\p_1^*+t\p_2^*)$ is given implicitly as $\omega\in\mathbb{R}$ such that
$$ P_G(\overline{\p_1^*+t\p_2^* - \omega}) = 0.$$
Given a potential $\bar{\p}:\bar{S}_A^{\mathbb{N}} \to \mathbb{R}$ consider the Ruelle operator 
$(L_{\bar{\p}}f)(x) = \sum_{\bar{\sigma}(\bar{y})=\bar{x}} e^{\bar{\p}(\bar{y})} f(\bar{y})$.
Let $\l(L_{\bar{\p}})$ denote the biggest eigenvalue of $L_{\bar{\p}}$.
By (2)-(4) and Lemma 4 in \cite{Sar00}, the operator $L_{\overline{\p_1^*-P(\p_1^*)}}$ has a spectral gap.
By standard results from perturbation theory of linear operators, \cite{Ka}, the same is true for all operators $L$
in some neighborhood $\mathcal{O}$ of $L_{\overline{\p_1^*-P(\p_1^*)}}$.
In addition, $L \to \l(L)$ is holomorphic, and if  $\bar{\p}:\bar{S}_A^{\mathbb{N}} \to  \mathbb{R}$ is such that $L_{\bar{\p}}\in \mathcal{O}$,
then $\l(L_{\bar{\p}})=\exp(P_G(\bar{\p}))$.

By (1),(2),(5) and the proof of Lemma 6 in \cite{Sar00}, $(z,\omega) \to L_{\overline{\p_1^* + z\p_2^*-\omega}}$ is analytic in a complex neighborhood of $(0,P(\p_1^*))$.
Consequently $(z,\omega) \to \log \lambda (L_{\overline{\p_1^* + z\p_2^*-\omega}})$ is holomorphic in a neighborhood of $(0,P(\p_1^*))$.
The lemma follows by the Complex Implicit Function Theorem (\cite{Boch}, p.39).
\end{proof}

Let $\p_1,\p_2: S_{A}^{\mathbb{Z}} \to \mathbb{R}$ be two potential functions, $a\in S$ a fixed state, $\left( \bar{S}_A^{\mathbb{Z}},\bar{\sigma} \right)$
the induced system on $[a]$, and $\bar{\p}_1,\bar{\p}_2 : \bar{S}_A^{\mathbb{Z}} \to \mathbb{R}$ the induced potentials.

\begin{lemma}\label{le2}
Assume that $\bar{\p}_1,\bar{\p}_2$ have summable variations. Then there exist functions $u_1,u_2 : S_A^{\mathbb{Z}} \to \mathbb{R}$ which are uniformly bounded on $[a]$ and such that the functions defined by:
$$ 
\Psi_l := \p_l + u_l\circ\sigma_A - u_l, ~~~~l=1,2  
$$
satisfy:
\begin{enumerate}
\item $\bar{\Psi}_l = \bar{\p}_l + u_l\circ \bar{\pi}\circ\bar{\sigma} - u_l\circ \bar{\pi}$, 
where $ \bar{\pi}: \bar{S}_A^{\mathbb{Z}}\to S_A^{\mathbb{Z}}$ is the canonical projection;
\item $\Psi_l(x)=\Psi_l(x')$, and $\bar{\Psi}_l(\bar{x})=\bar{\Psi}_l(\bar{x}')$ whenever $x_i=x_i', ~\bar{x}_i=\bar{x}_i'$ for all $i\ge 0$;
\item $P(\Psi_l)=P(\p_l)$, and $P_G(\bar{\Psi}_l)=P_G(\bar{\p}_l)$;
\item if $\bar{\p}_l$ is H\"{o}lder, then so is $\bar{\Psi}_l$;
\item if $\p_1,\p_2$ satisfy Condition (5) in Theorem \ref{presshift}, then so do $\Psi_1,\Psi_2$;
\item if $\bar{\p}_l$ satisfies Conditions (3) and (4) in Theorem \ref{presshift}, then so does 
$\bar{\Psi}_l$. 
\end{enumerate}
\end{lemma}

\begin{proof}
For any $x_0\in S$ pick a sequence $\left( r_{x_0,i} \right)_{i=-\infty}^{-1} \subset S$ such that: $$A_{r_{x_0,i}r_{x_0,i+1}}=1\text{ for all }i<-1
\text{ and } A_{r_{x_0,-1}x_0}=1.$$
Define $r: S_A^{\mathbb{Z}} \to S_A^{\mathbb{Z}} $ by setting $(r(x))_i = x_i$ for $i\geq 0$, and
$(r(x))_i=r_{x_0,i}$ for $i<0$.
Let 
$$ u_l(x) := \sum_{j=0}^{\infty} \p_l(\sigma_A^j(x))- \p_l(\sigma_A^j(r(x))).  $$
For $x\in [a]$ we have that
$$|u_l(x)| = | \sum_{j=0}^{\infty} \p_l(\sigma_A^j(x))- \p_l(\sigma_A^j(r(x)))|
 =  | \sum_{j=0}^{\infty} \bar{\p}_l(\bar{\sigma}^j(\bar{x}))- \bar{\p}_l(\bar{\sigma}^j(\overline{r(x)}))| $$
$$ \leq \sum_{j=1}^{\infty} Var_j \bar{\p}_l <\infty,$$
so that $u_l$ is uniformly bounded on $[a]$.
Observe that
$$ u_l(\sigma_A(x)) - u_l(x) = -\p_l + \sum_{j=0}^{\infty}  \p_l(\sigma_A^j(r(x))) -  \p_l(\sigma_A^j(r(\sigma_A(x)))).   $$
Therefore $\Psi_l = \p_l + u_l\circ\sigma_A - u_l$ depends only on positive sides of sequences.
In addition $\bar{\Psi}_l = \bar{\p}_l + u_l\circ \bar{\pi}\circ\bar{\sigma} - u_l\circ \bar{\pi}$, so that $\bar{\Psi}$ depends only on the positive sides of sequences.

To see that  $P(\Psi_l)=P(\p_l)$ note that for any $\sigma_A$-invariant measure $\mu$ we have that
\begin{align*}
 h_{\mu}(\sigma_A) + \int \Psi_l d\mu &=  h_{\mu}(\sigma_A) + \int \left( \p_l +  u_l\circ\sigma_A - u_l\right)  d\mu  \\
 &=  h_{\mu}(\sigma_A) + \int \p_l d\mu +  \int u_l\circ\sigma_A d\mu - \int u_l  d\mu  \\
 &=  h_{\mu}(\sigma_A) + \int \p_l d\mu +  \int u_l d\mu - \int u_l  d\mu \\   
 &= h_{\mu}(\sigma_A) + \int \p_l d\mu .
 \end{align*}

To see that   $P_G(\bar{\Psi}_l)=P_G(\bar{\p}_l)$ note that for any $\bar{x}\in \bar{S}^{\mathbb{Z}}_{A}$ with $\bar{\sigma}^n(\bar{x})=\bar{x}$ we have that
$$ \sum_{j=0}^{n-1} u_l(\bar{\pi}(\bar{\sigma}^{j+1}(\bar{x}))) - u_l(\bar{\pi}(\bar{\sigma}^j(\bar{x}))) =   u_l(\bar{\pi}(\bar{\sigma}^{n}(\bar{x}))) - u_l(\bar{\pi}(\bar{x})) =0. $$

(4) is proved in (Lemma 3.3,\cite{ind}).
 (5) follows from the fact that $u_l$ is bounded on $[a]$.
 (6) follows because $u_l$ is bounded on $[a]$ and  $P(\Psi_l)=P(\p_l)$.
\end{proof}

\begin{proof}[Proof of Theorem \ref{presshift}]
Consider functions $\p_1,\p_2$ satisfying the hypothesis of Theorem \ref{presshift}.
For $l=1,2$ let $\Psi_l^*: S_A^{\mathbb{N}} \to \mathbb{R}$ be defined by
$\Psi_l^*:= \Psi_l((\pi^*)^{-1}(x))$, where $\pi^* : S^{\mathbb{Z}} \to S^{\mathbb{N}}$ is the canonical projection and
$\Psi_l $ is defined in Lemma \ref{le2}. The function $\Psi_l^*$ is well defined since $\Psi_l$ only depends on the positive part of the sequence.
In addition, $\Psi_1^*$ and $\Psi_2^*$ satisfy the hypothesis of Lemma \ref{le1}.
Therefore $t \to P(\Psi_1^* + t\Psi_2^*)$ is real analytic.
We have that
\begin{align*}
 P(\Psi_1^* + t\Psi_2^*) &= P(\Psi_1 + t\Psi_2)\\ 
 &= P( \p_1 + u_1\circ\sigma_A - u_2 +t \p_2 + t(u_2\circ\sigma_A - u_2))   \\
  &=P( \p_1  +t \p_2 + \tilde{u}\circ\sigma_A - \tilde{u}), 
  \end{align*}
where $\tilde{u}:= u_1+t u_2$ is bounded on $[a]$.
By the argument presented in Lemma \ref{le2} we have that
$ P( \p_1  +t \p_2 + \tilde{u}\circ\sigma_A - \tilde{u}) =  P( \p_1  +t \p_2)$
and $t \to P( \p_1  +t \p_2)$ is real analytic on $(-\epsilon,\epsilon)$.
\end{proof}

\subsubsection{Proof of Theorem \ref{thm:pressure}}

Theorem \ref{thm:pressure} is an immediate consequence of Theorem \ref{presshift} and the following lemma.

For $l=1,2$ denote $\p_l=\hat{\varphi}_l\circ \pi_1$, where $\pi_1$ was introduced in Section \ref{sec:tow}.
Let $\bar{\p}_l$ be the corresponding induced potential on $[\tilde{X}]$.
We have the following.

\begin{lemma}
\begin{enumerate}

\item If $\varphi_1,\varphi_2$ satisfy $(1)-(3)$ in Theorem \ref{thm:pressure} then $\p_1, \p_2$ satisfy $(1)$ in Theorem \ref{presshift},\\

\item If $\varphi_1$ satisfies $(2)$ in Theorem \ref{thm:pressure} then $\p_1$ satisfies $(3)$ and $(4)$ in Theorem \ref{presshift},\\

\item If $\varphi_1,\varphi_2$ satisfy $(1)-(2)$ in Theorem \ref{thm:pressure} then $\p_1, \p_2$ satisfy $(5)$ in Theorem \ref{presshift}.\\

\end{enumerate}  
\end{lemma}

\begin{proof}
We start with the proof of the first statement. By Theorem 4.6 in \cite{ind} we have that $P_G(\overline{\varphi_1+t \varphi_2 +c})<\infty$ and 
$P_G(\overline{\varphi_1 +c})<\infty$. By Theorem 4.2 in \cite{ind} this implies that $P(\varphi_1+t \varphi_2 +c)<\infty$ and 
$P(\varphi_1+c)<\infty$. Since $c\in\mathbb{R}$ is a finite number, we obtain that $P(\varphi_1+t \varphi_2)<\infty$ and 
$P(\varphi_1)<\infty$. The fact that $P_G(\overline{\varphi_1+t \varphi_2 - P(\varphi_1 + t\varphi_2)})=0$ follows by Theorem 4.6 in \cite{ind}.

Assume now that  $\varphi_1 $ satisfies $(P4)$. This immediately implies that $\bar{\p}_1$ satisfies Condition $(3)$ in Theorem \ref{presshift}.
We also obtain that $P_G(\varphi_1^+)=P_G(\p_1^+)=0$ and there exists a unique Gibbs measure $\nu$ for $\varphi_1^+$.
Fix $J\in S$. We then have,

$$  \sum_{n\geq 1} Z_n \left( \overline{\p_1-P(\p_1)}, [J]  \right)  = 
  \sum_{n\geq 1} ~ \sum_{\bar{x}\in J, \tilde{f}^n(\bar{x})=\bar{x}} e^{( \varphi_1^+(\bar{x}) )_n}  $$
  $$ \geq C \sum_{n\geq 1} \sum_{[JJ_1\ldots J_{n-1}J]}   \nu([JJ_1\ldots J_{n-1}J]) =  C \sum_{n\geq 1}   \nu([J]) = \infty. $$
  
  On the other hand,
  
 $$\sum_{n\geq 1} n Z_n^* \left( \overline{\p_1-P(\p_1)}, [J]  \right) = \sum_{n\geq 1} n  \sum_{\bar{x}\in J, \tilde{f}^n(\bar{x})=\bar{x} \\ \tilde{f}^k(\bar{x})\notin J \text{ for } k<n}  e^{( \varphi_1^+(\bar{x}) )_n} $$
 $$ \leq C \sum_{n\geq 1} n  \sum_{[JJ_1\ldots J_{n-1}J], J_i\neq J}  \nu([JJ_1\ldots J_{n-1}J])   $$
 $$   =  \sum_{n\geq 1} n \nu( \{x\in J ~|~ n_J(\bar{x})=n     \}) < \infty.$$
 Here $ n_J(\bar{x})$ denotes the first return time of $\tilde{f}(\bar{x})$ to $J$.
 This proves that $\p_1$ satisfies Condition $(4)$ in Theorem \ref{presshift}.

By (2) in Theorem \ref{thm:pressure} for each $t\in [-\frac{\epsilon_0}{2},\frac{\epsilon_0}{2}]$ there exists $\epsilon_t>0$ such that
$$ \sum_{J\in S} \tau(J) \sup_{x\in J} \exp( \overline{\varphi_1(x) + t |\varphi_2(x)|- P_L(\varphi_1(x) + t|\varphi_2(x)|) + \epsilon_t}) < \infty.  $$
Let $\tilde{\epsilon}:=\min\{\epsilon_t | t\in  [-\frac{\epsilon_0}{2},\frac{\epsilon_0}{2}] \}$ and let $|t|$ be small enough so that 
$- P_L(\varphi_1(x) + t|\varphi_2(x)|) + \tilde{\epsilon} > - P_L(\varphi_1)$.
Set $r:= \exp(- P_L(\varphi_1(x) + t|\varphi_2(x)|) + \tilde{\epsilon} )$.
We have that,

$$ \sum_{n\geq 1} n r^n Z_n^*(\p_1 + t |\p_2|, \tilde{X}) \leq  \sum_{n\geq 1} n r^n  \sum_{\tau(J)=n} \sup_{x\in J} \exp( \overline{\varphi_1(x) + t |\varphi_2(x)|})$$ 
$$= \sum_{n\geq 1} n  \sum_{\tau(J)=n} \sup_{x\in J} \exp( \overline{\varphi_1(x) + t |\varphi_2(x)|+ \ln r }) < \infty. $$
\end{proof}

\subsection{Proof of Theorem \ref{appl: Young ergprop}}
 The first statement is proved as Theorem 7.7 in \cite{ind}. Moreover, it is proved that the measure $\mu_t$ has exponential tail.  Note that $f$ can be  represented as a map with an inducing scheme by Proposition 6.2. in \cite{ind}. 
Statements (4a), (4b), and (4d) of Theorem \ref{thm: manifold} imply that  $\mu_t$ is mixing, has exponential decay of correlations, and satisfies the CLT. Once $\mu_t$ is mixing, then by Statement (3) of Theorem \ref{thm: manifold}, it is Bernoulli.

\end{document}